\documentclass[11pt,final]{article}
\usepackage{amsmath,amsthm,amsfonts,amssymb,graphicx,hyperref,enumerate,psfrag}
\usepackage[a4paper,margin=2.5cm]{geometry}
\usepackage{fullpage}
\usepackage{enumitem}
\usepackage{xcolor}

\usepackage[utf8]{inputenc} 

\newtheorem{theorem}{Theorem}[section]

\newtheorem{lemma}[theorem]{Lemma}

\newtheorem{remark}[theorem]{Remark}
\newtheorem{remarks}[theorem]{Remarks}
\newtheorem{proposition}[theorem]{Proposition}

\newtheorem{corollary}[theorem]{Corollary}

\newcommand{\dE}{\mathbb {E}}
\newcommand{\dP}{\mathbb {P}}
\newcommand{\dZ}{\mathbb {Z}}
\newcommand{\dN}{\mathbb {N}}
\newcommand{\dR}{\mathbb {R}}

\newcommand{\bP}{\mathbf{P}}
\newcommand{\bE}{\mathbf{E}}

\newcommand{\cA}{\mathcal {A}}
\newcommand{\cB}{\mathcal {B}}
\newcommand{\cC}{\mathcal {C}}

\newcommand{\cE}{\mathcal {E}}

\newcommand{\cG}{\mathcal {G}}

\newcommand{\cM}{\mathcal {M}}
\newcommand{\cN}{\mathcal {N}}
\newcommand{\cO}{\mathcal {O}}
\newcommand{\cP}{\mathcal {P}}

\newcommand{\cS}{\mathcal {S}}
\newcommand{\cT}{\mathcal {T}}

\newcommand{\cW}{\mathcal {W}}
\newcommand{\cX}{\mathcal {X}}
\newcommand{\cZ}{\mathcal {Z}}

\newcommand{\h}{\mathfrak{h}}

\newcommand{\phid}{\varphi_{\mathbb{T}_d}}

\newcommand{\Td}{\mathbb{T}_d}

\newcommand{\gtd}{G_{\Td}}

\newcommand{\Cov}{\text{Cov}}
\newcommand{\Eh}{E_{\phid}^{\geq h}}

\newcommand{\Ch}{\cC_{\circ}^{h}}

\newcommand{\Chplus}{\cC_{\circ}^{h,+}}

\newcommand{\Chplusp}{\cC_{\circ,p}^{h,+}}

\newcommand{\oc}{\overline{\circ}}

\newcommand{\stdom}{\underset{st.}{\leq}}

\newcommand{\bPc}{\bP^{\Ch}}
\newcommand{\dPd}{\dP^{\Td}}
\newcommand{\dPds}{\dP^{\Td,\infty}}

\newcommand{\sP}{\mathsf{P}}
\newcommand{\sE}{\mathsf{E}}
\newcommand{\sPr}{\mathsf{P}^{{renew}}}
\newcommand{\sEr}{\mathsf{E}^{{renew}}}

\newcommand{\Ci}{C_1}
\newcommand{\Cii}{C_2}
\newcommand{\Cv}{C_3}
\newcommand{\Cvi}{C_4}
\newcommand{\Cvii}{C_5}
\newcommand{\Cviii}{C_6}
\newcommand{\Cix}{C_7}
\newcommand{\Cx}{C_{8}}
\newcommand{\Cxi}{C_{9}}

\title{Speed of the random walk on the supercritical Gaussian Free Field percolation on regular trees}
\author{Guillaume Conchon-\,\hspace{-0.3mm}-Kerjan\thanks{Department of Mathematics, King's College London, W2CR 2LS London, UK.\newline Email: \texttt{guillume.conchon-kerjan@kcl.ac.uk}}  
}
\date{}
\begin{document}

\maketitle
\begin{abstract}
In this paper, we study the random walk on a supercritical branching process with an uncountable and unbounded set of types supported on the $d$-regular tree $\Td$ ($d\geq 3$), namely the cluster $\Ch$ of the root in the level set of the Gaussian Free Field (GFF) above an arbitrary value $h\in (-\infty, h_{\star})$. The value $h_{\star}\in (0,\infty)$ is the percolation threshold; in particular, $\Ch$ is infinite with positive probability.
\\
We show that on $\Ch$ conditioned to be infinite, the simple random walk is ballistic, and we give a law of large numbers and a Donsker theorem for its speed. 
\\
To do so, we design a renewal construction that withstands the long-range dependencies in the structure of the tree. This allows us to translate underlying ergodic properties of $\Ch$ into regularity estimates for the random walk. 
\end{abstract}

\section{Introduction}

\subsection{Context and motivation}
\noindent
Perfoming random walks on large random structures allows to reveal some of their geometrical properties, such as their expansion, their connectivity, the presence of traps or bottlenecks, etc. A standard class of such structures are Bernoulli percolation clusters on $\dZ^d$, for which there exists nowadays a rich litterature. It is known that on supercritical clusters, the simple random walk (SRW) is diffusive, and the scaling limit is a Brownian motion~\cite{Barlow, BergerBiskup, SidoraviciusSznitman}. The critical regime is much more delicate and still open. In high dimensions, the scaling of the random walk has been identified~\cite{KozmaNachmias}, and it is conjectured that the scaling limit is a Brownian motion on an integrated super brownian excursion~\cite{BenArousCabezasFribergh, Croydon}. 
\\
Random trees provide a more accessible setting that can give an intuition about percolation on $\dZ^d$ (in particular for large $d$), and which is also interesting in its own right. On Galton-Watson trees, the scaling limit of the random walk has been identified in the critical case~\cite{CroydonKumagai, Kesten}, and it was proved in the supercritical case that the random walk has a positive speed~\cite{LPPergodic}. Since then, much attention has been devoted to biased random walks, i.e. the SRW with a bias towards (or away from) the root of the tree, see for instance~\cite{Aidekon,BenArousFGH,Bowditch,CollevechioHolmesKious,CroydonFriberghKumagai,LPP96}.
\\
\\
In the last two decades, a percolation model with long-range correlations has taken a prominent place in probability, namely the level-set percolation of the Gaussian Free Field (GFF). On an infinite connected graph where the random walk is transient, the GFF is a Gaussian process whose covariance is given by the Green function. Hence, the GFF carries a lot of information on the structure of $\cG$ and on the behaviour of random walks, giving a base motivation for its study. It also has deep structural link with other probabilistic objects, such as 
local times of random walks~\cite{eisenbaum2000,Lupu,SabotTarres} and random interlacements~\cite{Sznitman12RIandGFF,sznitman2012}.  
It has been heavily studied on $\dZ^d$~\cite{DrewitzPrevostRodriguez,DPRCritExp,DuminalGFF2,Muirhead,RodriguezSznitman}, where there is a non-trivial sharp phase transition. 
\\
Recently, level-set percolation has also been subject to much attention on regular trees~\cite{SZ2016,ACregultrees} and Galton-Watson trees~\cite{AS2018,DrewitzGalloPrevost,CernyLocher}. The GFF percolation yields a branching process with a natural notion of fitness: the fitness of a vertex is given by the value of the GFF. It is partly hereditary, and in level-set percolation, only strong enough individuals survive. While the absence of cycles removes some difficulties inherent to finite-dimensional lattices, the long-range dependencies of the GFF, as well as the uncountable and unbounded range of types, bring significant additional challenges compared to Galton-Watson trees, and many standard techniques break down.
\\
Even on regular trees, not much is known on the supercritical regime, apart from the fact that the percolation cluster from the root, when infinite, grows at an exponential rate. The SRW on this root cluster provides a good case study for a random walk on a correlated environment.  The aim of this paper is to bring a thorough understanding of this walk, which in turns provides fine properties on the structure of the percolation cluster. 

\subsection{Setting}\label{subsec:defs}
\noindent
In all this work, we fix an integer $d\geq 3$. We denote $\Td$ the infinite $d$-regular tree rooted at an arbitrary vertex $\circ$. 

\subsubsection{The Gaussian Free Field (GFF) on $\Td$.}

\noindent
The GFF $\phid$ on $\Td$ is a centred Gaussian field $(\phid(x))_{x\in \Td}$ indexed by the vertices of $\Td$, and with covariances given by the Green function $\gtd$ on $\Td$: for all $x,y\in \Td$, $\text{Var}(\phid(x),\phid(y))=\gtd(x,y)$. Recall that 
\[
\gtd(x,y)= \bE_x^{\Td}\left[\sum_{k\geq 0}\mathbf{1}_{\{X_k=y\}}\right]\]
where $(X_k)_{k\geq 0}$ is a (discrete-time) SRW on $\Td$. For a graph $\cG$ and $x\in \cG$, we denote $\bP_{x}^{\cG}$ and $\bE_{x}^{\cG}$ the probability and expectation associated to a SRW $(X_k)_{k\geq 0}$ such that $X_0=x$. 
\\
For $h\in \dR$, let $\Eh:=\{x\in \Td \, \vert \, \phid(x)\geq h\}$ be the \textbf{level-set above $h$}. Let $\Ch$ be the connected component of $\Eh$ containing the root $\circ$. For $x\in \Td$, let $\vert x\vert$ be its height, i.e.~its distance to $\circ$. For $k\geq 0$, denote $\cZ_k^h:=\{x\in \Ch, \vert x\vert=k\}$ the $k$-th generation of $\Ch$. 

\subsubsection{Probability spaces}\label{subsubsec:probaspaces}

\noindent
Write $\dPd$ and $\dE^{\Td}$ for the probability and expectation w.r.t.~$\phid$, $\dPds$ for $\dPd$ conditionally on $\vert \Ch\vert=\infty$, and $\dPd_a$ for $\dPd$ conditionally on $\phid(\circ)=a$, for $a\in \dR$.
\\
The main object of interest in this paper is the SRW on $\Ch$ started at the root. We write $\bP^{\Ch}:=\bP^{\Ch}_{\circ}$ for the quenched probability of the SRW. Our most important result (Theorem~\ref{thm:LLN}) will be stated with respect to the annealed probability measure 
\begin{equation}\label{eqn:defPannealed}
\sP^{h,\infty}(\cdot)=\int \bP^{\Ch}(\cdot) d\dP^{\Td,\infty}. 
\end{equation}
Throughout the paper, we will need a range of auxiliary probability measures. First, we write $\sP^{h}(\cdot)=\int \bP^{\Ch}(\cdot) d\dP^{\Td}$.  
\\
Second, as root of $\Td$, $\circ$ has one more child than any other vertex. For technical reasons, we will need to break this particularity. Let $\oc$ be an arbitrary neighbour of $\circ$. Let $\Td^+$ be the largest subtree rooted at $\circ$  that does not contain $\oc$, so that every vertex of $\Td^+$ has $d-1$ children in $\Td^+$. Write $\Chplus:=\Ch\cap \Td^+$. 
Denote $\sP^{h,+}(\cdot)$ and $\sP^{h,+,\infty}(\cdot):=\sP^{h,+}(\cdot\,\vert\,\vert \Chplus\vert=\infty)$ the corresponding annealed probabilities when the SRW is instead on $\Chplus\cup \{\oc\}$ (hence the edge $\oc$ stays open even if $\phid(\oc)<h$). 
\\
Third, we will also need to condition w.r.t.~the initial value of the GFF. We let
\begin{equation}\label{eqn:defPetcannealed}
\sP^{\star}_a(\cdot):=\sP^{\star}(\cdot \,\vert \, \phid(\circ)=a)
\end{equation}
for  $\sP^{\star}=\sP^{h,\infty},\sP^{h,+}$ or $\sP^{h,+,\infty}$ and for $a\in \dR$. Denote $\sE^{\star}$, and $\sE^{\star}_a$ the corresponding expectations. 


\subsection{Result}\label{subsec:results}
It is known that $\Ch$ undergoes a phase transition  (Theorems 4.3 and 5.1 in~\cite{ACregultrees}). If we define
\begin{equation}\label{eqn:etahdef}
\eta(h):=\dPd(\vert \Ch\vert =\infty),
\end{equation}
then there exists $h_\star=h_\star(d)>0$ such that when $h>h_\star$,  $\eta(h)=0$, and we even have $\limsup_{k\rightarrow +\infty} k^{-1}\log\dPd(\vert \Ch\vert \geq k)<0$. On the contrary, when $h<h_\star$, $\eta(h)>0$ and there exists $\lambda_h>1$ so that $\lim_{k\rightarrow +\infty}\dPd(\lambda_h^k/k^2\leq \vert\cZ_k^h \vert\leq k\lambda_h^k )=\eta(h)$. It was recently shown that $\eta$ is continuous at $h_\star$~\cite{CernyLocher}.

\noindent
In words, in the supercritical regime $h<h_{\star}$, $\Ch$ grows at an exponential rate when it is infinite (which happens with positive probability). Refinements on the growth of $\Ch$ have been established in~\cite{GckGffPublished} (see Section~\ref{sec:basic}),  showing further structural similarities between $\Ch$ and supercritical Galton-Watson trees. 
\\
A natural question then is whether these similarities still hold for finer properties, in particular for the behaviour of the SRW on $\Ch$.  
On a supercritical Galton-Watson tree (with finite mean offspring distribution) conditioned to survive, it is known that the SRW has a positive speed~\cite{LPPergodic}. It is not obvious that the same will hold on $\Ch$, even when knows that the size of its generations asymptotically grow at some rate $\lambda_h>1$. In particular, contrary to Galton-Watson trees, GFF level-sets have long-range dependencies. These may cause traps that prevent the walk from being ballistic, as well as inhomogeneities in the tree structure that would prevent (a.s., or at least with positive probability) the ratio $\vert X_k\vert /k$ to converge to a limit.
\\
The main finding of this paper is that such abnormal behaviour does not occur: we show the existence of a speed that a.s.~does not depend on the realization of $\Ch$.

\begin{theorem}[Strong LLN and annealed CLT]\label{thm:LLN}
For every $h<h_{\star}$, there exists constants $s_h,\sigma_h>0$ such that if $(X_k)_{k\geq 0}$ is a SRW on $\Ch$ started at $\circ$, then $\sP^{h,\infty}$-almost surely,
\begin{equation}\label{eqn:thmLLN}
\left(\frac{\vert X_{\lfloor kt\rfloor}\vert }{k}\right)_{0\leq t\leq 1}{\longrightarrow}\,(s_ht)_{0\leq t\leq 1}
\end{equation}
and under~$\sP^{h,\infty}$,
\begin{equation}\label{eqn:thmCLT}
\left(\frac{\vert X_{\lfloor kt \rfloor}\vert -s_hkt}{\sigma_h\sqrt{k}}\right)_{0\leq t\leq 1}\overset{(d)}{\longrightarrow} (B_t)_{0\leq t\leq 1}
\end{equation}
where $B$ is a standard real Brownian motion and both convergences hold w.r.t.~the Skorokhod metric on $[0,1]$ as $k\rightarrow +\infty$. 
\\
The results hold under $\sP^{h,\infty}_a$ instead of $\sP^{h,\infty}$, for any $a\geq h$.
\end{theorem}

\noindent
Remark that~\eqref{eqn:thmLLN} yields immediatly a quenched LLN for the speed of the random walk (recall~\eqref{eqn:defPannealed}): $\dPds$-almost surely, $\bPc$ is such that~\eqref{eqn:thmLLN} holds.

\subsection{Proof strategy}\label{subsec:proofstrategy}

\noindent
Throughout the paper, we use an equivalent definition of $\phid$ which displays its branching nature on $\Td$: for every vertex $x$ on $\Td$, conditionally on $\phid(x)$, the values of $\phid$ on the children of $x$ are i.i.d.~with an explicit gaussian distribution (Proposition~\ref{prop:recursivegfftrees}). This gives a recursive construction of $\phid$, starting from the root and going from one generation to the next.
\\
\\
\textbf{I. Transience.} The first step is to show the transience of the SRW (Section~\ref{sec:transience}). For technical purposes, we prove the transience on $\Chplus$ instead of $\Ch$ (recall Section~\ref{subsubsec:probaspaces}). 
We also need to quantify uniformly this transience. For $\delta >0$, say that a rooted tree $T$ with root $r$ is $\delta$\textbf{-transient} if $\bP_r^T(\forall k>0, \, X_k\neq r)\geq \delta$, that is, the SRW on $T$ started at the root has a probability at least $\delta$ to never return to its initial location. For any $\delta >0$, and $a\in \dR$, let 
\begin{equation}\label{eqn:qhdeltadef}
q_{h,\delta}(a):=\dP_a^{\Td}(\Chplus\text{ is not $\delta$-transient}).
\end{equation}

\begin{proposition}\label{prop:transienceintro}
For $\delta_0$ small enough (that only depends on $d$ and $h$), there exists $\varepsilon >0$ so that for every $\delta \in (0,\delta_0)$, for every $a\geq h$, 
$q_{h,\delta}(a)< 1-\varepsilon$.
\end{proposition}
\noindent
The proof goes by showing that the branching number of $\Chplus$ is larger than $1$, using precise estimates on the growth rate from Section~\ref{sec:basic}. This classically implies the transience. 
\\
\\
Then if one could find an invariant measure for the environment seen from the random walker (i.e. we re-root $\Td$ at $X_n$, for $n\geq 0$), one could derive a law of large numbers for the speed as was done for the SRW on supercritical Galton-Watson trees in~\cite{LPPergodic}, and for the biased random walk on the same trees in~\cite{Aidekon}. Unfortunately, we have not been able to find such a measure, as the fact that vertices carry random variables ($\phid$) significantly complicates the matter, and somehow breaks the reversibility. Besides, proving only the existence of this measure would not be enough: one also needs to ensure that $s_h>0$. We discuss our attempts in Section~\ref{subsec:invarmeasurenomore}.
\\
\\
\textbf{II. Renewal.} Instead, we prove a stronger assertion than ballisticity, namely that the SRW on $\Chplus$ conditioned to be infinite has renewal times with stretched exponential moments.  
We say that $R\geq 0$ is a \textbf{renewal time} of $(X_k)$ if $\max_{k\leq R-1}\vert X_{k}\vert<\min_{k\geq R}\vert X_{k}\vert$. In particular, the SRW goes through the edge $(X_{R-1},X_{R})$ only once. Let $(\tau_i)_{i\geq 1}$ denote the renewal times of $(X_k)$, with $\tau_i=+\infty$ if there are at most $i-1$ renewal times. Recall the definition of $\sP^{h,+,\infty}_a $ from~\eqref{eqn:defPetcannealed} as the annealed probability conditionally on $\Chplus$ being infinite and $\phid(\circ)=a$. We prove the following Proposition in Section~\ref{sec:renewal}.
\begin{proposition}\label{lem:expomomentsdiscovery}
There exist constants $\Ci,\Cii >0$ such that for every $a\geq h$ and $k\geq 1$, 
$$\sP^{h,+,\infty}_a(\tau_1\geq k)\leq \Ci e^{-\Cii k^{1/6}}.$$
\end{proposition}

\noindent
The proof of Proposition~\ref{lem:expomomentsdiscovery} relies on a crucial structural fact: for almost every infinite realization of $\Chplus$, on every finite path starting grom the root, a positive proportion of the vertices offer at least two uniformly transient subtrees for the random walk (Lemma~\ref{lem:grimmettkesten}). Thus, if $\vert X_i\vert=k$, the probability that $(X_j)_{j\geq i}$ returns to $\circ$ decays exponentially with $k$. This fact also holds on infinite supercritical Galton-Watson trees (Lemma 2.1 in~\cite{GrimmettKesten}), and can be used to prove a similar renewal property on these trees~\cite{Piau}.
\\
Then, Proposition~\ref{lem:expomomentsdiscovery} implies that $\sup_{a\geq h,\,i\geq 1}\sE^{h,+,\infty}_a(\tau_i)<\infty$. Using that $\vert X_{\tau_{i+1}}\vert-\vert X_{\tau_i}\vert \geq 1$ a.s., one could easily deduce ballisticity estimates - for instance $\sE^{h,+,\infty}[\vert X_k\vert /k]\geq ck$ for some constant $c>0$ and all $k$ large enough. 
\\
However, this does not automatically entail a law of large numbers, as contrary to the Galton-Watson case, the pieces of trajectory $(X_k)_{\tau_i< k\leq \tau_{i+1}}$ for $i\geq 1$ are not i.i.d.: more precisely, $(X_k)_{k> \tau_i}$ is independent of $(X_k)_{k\leq \tau_i}$ conditionally on $\phid(X_{\tau_i})$ and when rerooted at $X_{\tau_i+1}$, it has the law of $(X_k)_{k\geq 0}$ under 
\begin{equation}
\sPr_a:=\sP^{h,+,\infty}_{\phid(X_{\tau_i})}(\cdot\,\vert \, \forall k\geq 0, X_k\neq \oc ),
\end{equation}
see Remark~\ref{rem:renewalabsolutcont} and Proposition~\ref{prop:expomomentsrenewal}. The main issue is the regularity of $\sPr_{a}$ w.r.t.~$a$, in particular of the quantities $\sEr_{a}[\tau_1]$ and $\sEr_{a}[\vert X_{\tau_1}\vert ]$.
\\
\\
%
%
%
%
\textbf{III. Regularity of the renewal intervals.} 
Since renewal intervals are independent conditionally on the value of $\phid$ at the entrance of these intervals, we can decompose the trajectory of $(X_k)$ into a Markov chain that keeps track of the height and duration of renewal intervals, as well as the value of $\phid$ at the exit of the interval. A key point is that renewal intervals have light tails (stretched exponential), so that regularity properties of $\phid$ on $\Td$ (for instance, in spite of the long-range correlations, $\phid(x)$ has a uniform Gaussian tail for ever $x\in \Td$) can be translated to the sequence $(\phid(X_{\tau_i})_{i\geq 1}$, which governs the distribution of the renewal intervals.
\\
\\
In detail, for $i\geq 1$, let $W^{(\tau)}_i:=(X_{\tau_{i}}, \ldots, X_{\tau_{i+1}})$ be the trajectory of $(X_k)_{k\geq 0}$ on the $(i+1)$-th renewal interval, and let $T^{(\tau)}_{i}$ be the subtree from $X_{\tau_{i}}$ in $\Chplus$ of height $ \vert X_{\tau_{i+1}}\vert - \vert X_{\tau_{i}}\vert  $, on which the trajectory $W^{(\tau)}_i$ lives. Then, the sequence $(Y_i)_{i\geq 0}$  defined by
\begin{equation}\label{eqn:defYiMC}
Y_i:=(\phid(X_{\tau_{i+1}}),T^{(\tau)}_i,W^{(\tau)}_i  )
\end{equation}
is a Markov chain on the state space $\cX:=[h,+\infty)\times \cM$, where $\cM$ is, roughly, the set of couples $(T,W)$ where $T$ is a finite tree and $W$ a trajectory starting at the root of $T$ and ending at a vertex of maximal height (see~\eqref{eqn:defMset} for an exact definition). 
\\
We prove that this Markov chain is positive Harris recurrent - in particular, it has a (non-explicit) invariant measure, and that it is uniformly ergodic w.r.t.~the auxiliary drift function
\begin{equation}\label{eqn:Vpotentialdef}
V(Y):=\varphi(Y) + h(Y)^2+\tau(Y)^2
\end{equation} 
for $Y\in \cX$. 
In simple terms, we show that the sequence $(V(Y_i))_{i\geq 0}$ visits regularly a compact set $\cC$ of $[0,+\infty]$, and that from any $Y\in V^{-1}(\cC)$, the chain has probability at least $\alpha >0$ to 'forget its past' (respectively~\eqref{eqn:driftcondition} and~\eqref{eqn:doeblincondition} in Lemma~\ref{lem:Vpotential}). 
\\
Let us explain the reason why $(V(Y_i))_{i\geq 0}$ cannot stay for too long on high values. The duration (and thus the height) of the $(i+1)$-th renewal interval has stretched exponential bounds, uniformly in $i$ and $\phid(X_{\tau_{i+1}})$ (Proposition~\ref{prop:expomomentsrenewal}). Hence we have a very good control on the distribution of $h(Y_{i+1})^2+\tau(Y_{i+1})^2$. Moreover, the sequence $(\varphi(Y_i))_{i\geq 0}$ is attracted to low values, as per the following reasoning. If $\varphi(Y_i)=\phid(X_{\tau_i})$ is very high, then because of the Gaussian tails of $\phid$, $\varphi(Y_i)$ is likely to be the strict maximum of $\phid$, by a sizeable margin, on a large neighbourhood of $X_{\tau_i}$. Since renewal intervals are short, as mentioned just above, $X_{\tau_{i+1}}$ has a high chance to be in this neighbourhood, so that with large probability, we will have $\varphi(Y_{i+1})< c \varphi(Y_i)$ for some constant $c\in (0,1)$. 
\\
By an ergodic theorem from~\cite{MeynTweedie}, these regularity properties on $(Y_i)_{i\geq 0}$ (and thus on the sequences $(\tau_{i+1}-\tau_i)_{i\geq 0}$ and $(\vert X_{\tau_{i+1}}\vert -\vert X_{\tau_i}\vert )_{i\geq 0}$) are enough to ensure that $(\tau_k)$ and $(\vert X_{\tau_k}\vert ) $ satisfy a LLN and a CLT: 
\begin{proposition}\label{prop:CLThtau}
There exist constants $s_{h,\tau},s_{h,X}>0$ and $\sigma_{h,\tau},\sigma_{h,X}\geq 0$ so that for any $a\geq h$, under $\sPr_a$, 
\begin{equation}\label{eqn:CLThandtau}
\frac{\tau_k}{k}\overset{a.s.}{\longrightarrow} s_{h,\tau}\,\,\,  ;\,\,\,\frac{\tau_k - s_{h,\tau}k}{\sqrt{k}}\overset{(d)}{\longrightarrow} \cN(0,\sigma_{h,\tau}^2 ) \text{ and }
\end{equation}
\begin{equation}\label{eqn:CLThandtauX}
\frac{\vert X_{\tau_k}\vert }{k}\overset{a.s.}{\longrightarrow} s_{h,X}\,\,\,  ;\,\,\,\frac{\vert X_{\tau_k}\vert  - s_{h,X}k}{\sqrt{k}}\overset{(d)}{\longrightarrow} \cN(0,\sigma_{h,X}^2 ).
\end{equation}
as $k\rightarrow +\infty$, where convergence in distribution to $\cN(0,0)$ means convergence in probability to $0$.
\\
More generally, for any map $f:\cX\mapsto \dR$ such that $f^2(Y)\leq V(Y)$ for all $Y\in \cX$ (with $V$ defined in~\eqref{eqn:Vpotentialdef} above), the series $\sum_{i=0}^kf(Y_i)$ satisfies a LLN and a CLT (with adhoc constants $s_{h,f},\sigma_{h,f}\geq 0$) as $k\rightarrow +\infty$. 
\end{proposition}

\noindent
From there, we show a pointwise LLN and CLT for the SRW (Proposition~\ref{prop:LLNandCLT}), and  we conclude the proof of Theorem~\ref{thm:LLN} in Section~\ref{subsec:proofthmmain} via standard arguments.

\subsection{Open questions and related works}\label{subsec:openquestions}
\noindent
\textbf{GFF on Galton-Watson trees.} We believe that some of our arguments can be generalized from $\Td$ to supercritical Galton-Watson trees, up to a technical cost. In spite of the additional inhomogeneities, such trees have a.s.~a uniform exponential growth as described below Proposition~\ref{prop:expomomentsrenewal} (hence the Green function still decays exponentially fast with the distance between pair of vertices), and the fact that vertices have i.i.d.~offspring (hence disjoint parts of the tree are independent) brings some regularity. 
\\
In a recent paper~\cite{DrewitzGalloPrevost}, it was shown that for every offspring distribution with finite mean $m>1$, the critical threshold for GFF percolation is positive (solving a question from~\cite{AS2018}). A by-product of the proof, which relies on a clever construction using the links between the GFF and random interlacements via a Ray-Knight theorem, is that the SRW on $\Ch$ is transient in a non-trivial part of the supercritical regime (when $h$ is negative or close enough to 0). 
\\
\\
\textbf{Monotonicity of the speed and bias. }A natural question about Theorem~\ref{thm:LLN} is whether the map $h\mapsto s_h$ is monotonic. There does not seem to be an obvious answer. One shows easily that $\sup_{h<h_\star}s_h=(d-2)/d$, and that $\lim_{h\rightarrow -\infty}s_h=(d-2)/d$, which is the speed of the SRW on $\Td$ (in short, a classical martingale argument shows that $s_h$ cannot be larger than the speed of the SRW on a $\Td$ since no vertex has degree larger than $d$ in $\Ch$, and as $h\rightarrow -\infty$, the subtree of $\Ch$ seen by the SRW during the first renewal intervals is $d$-regular with high probability). One can conjecture that this convergence as $h\rightarrow-\infty$ is monotonic. On the other hand, one can conjecture that for $\varepsilon>0$ small enough, $h\mapsto s_h$ is decreasing on $[h_\star-\varepsilon,h_\star)$ and converges to $0$. 
\\
This is somewhat reminiscent of the variations of the speed of biased random walks on Galton-Watson trees w.r.t.~the bias away from the root, a topic that has been subject to much attention~\cite{Aidekon,BenArousFGH,Bowditch,CollevechioHolmesKious,CroydonFriberghKumagai,LPP96}. It is known that the speed is an increasing function of the bias when the latter is close to the critical value that makes the random walk recurrent, and that if the tree has leaves, the speed decreases to $0$ when the bias goes to infinity, since the random walk loses a considerable amount of time in traps. One could also investigate the possible variation profiles that one can obtain for the speed of biased random walks on $\Ch$, when $h$ spans $(-\infty, h_\star)$. 
\\
\\
\textbf{Critical GFF trees. } Very recently, it was shown that $\cC_\circ^{h_\star}$ is a.s.~finite~\cite{CernyLocher}. It would be interesting to give a proper definition for $\cC_\circ^{h_\star}$ conditioned to be infinite, and to investigate the behaviour of the SRW on it. The SRW should be recurrent, and if it is, does it exhibit the same fluctuations as the SRW on a critical Galton-Watson tree (with an offspring distribution having a finite second moment) conditioned to be infinite~\cite{Kesten}?

\subsection{Plan of the paper}\label{subsec:plan}

In Section~\ref{sec:basic}, we introduce the recursive construction of $\phid$ on $\Td$, the intergenerational operator $L_h$ and other related objects. We also state several technical results on the exponential growth of $\cZ_k^h$. 
In Section~\ref{sec:transience}, we establish the transience, proving Proposition~\ref{prop:transienceintro}. In Section~\ref{sec:renewal}, we show the existence of renewal interval with stretched exponential moments, proving Proposition~\ref{lem:expomomentsdiscovery}. In Section~\ref{sec:speed}, we establish regularity of the renewal intervals (Proposition~\ref{prop:CLThtau}) and prove Theorem~\ref{thm:LLN}. 

\subsection{Further definitions and conventions}
Trees in this paper are locally finite and undirected. 
For any tree $T$, denote $d_T$ the standard graph distance on its vertex set. For every vertex $x$ and integer $R\geq 0$, we define $B_T(x,R):=\{y, \, d_T(x,y)\leq R\}$ and $\partial B_T(x,R+1):=B_T(x,R+1)\setminus B_T(x,R)$. 
\\
If the tree is rooted at a distinguished vertex $\rho$, the \textbf{height} $\mathfrak{h}_T(x)$ of a vertex $x$  is $d_T(\circ,x)$. For simplicity, we write $\vert x\vert$ when $x$ is is in $\Td$ or a subtree rooted at $\circ$. The \textbf{ray of $x$}, denoted $\xi_x$, is the unique injective path from $\circ$ to $x$ (a \textbf{path} being a sequence of vertices such that any two consecutive vertices are neighbours). 
\\
The \textbf{offspring} of $x$ is the set $\cO_x$ of vertices $y$ such that $x\in \xi_y$. The tree induced by these vertices is the \textbf{subtree from $x$}. For $r\geq 0$, the \textbf{$r$-offspring} $\cO_x(r)$ of $x$ is its offspring at distance $r$ of $x$, and its \textbf{offspring up to generation $r$} is its offspring at distance at most $r$. If $y$ is in the $1$-offspring of $x$, then $y$ is a \textbf{child} of $x$, and $x$ is its \textbf{parent}. In this case, write $x=\overline{y}$.
\\
\\
Numbered constants $C_1 , C_2, \ldots$ only depend on $d$ and $h$, whereas other constants such as $c,c',\ldots$ may depend from other parameters, and change from one line to the next in the same computation.

\section{A branching process with an exponential growth}\label{sec:basic}

\subsection{An intergenerational operator}
There is an alternate definition of $\phid$, starting from its value at $\circ$ and expanding recursively to its neighbours. It shows that $\Ch$ is an infinite-type branching process, the type of a vertex $x$ being $\phid(x)$.

\begin{proposition}[\textbf{Recursive construction of the GFF},\cite{ACregultrees}]\label{prop:recursivegfftrees}
Define a Gaussian field $\varphi$ on $\Td$ as follows: let $(\zeta_y)_{y\in\Td}$ be a family of i.i.d. $\cN(0,1)$ random variables. Let $\varphi(\circ):=\sqrt{\frac{d-1}{d-2}}\zeta_{\circ}$. For every $y\in \Td\setminus \{\circ\}$, define recursively $\varphi(y):=\sqrt{\frac{d}{d-1}}\zeta_y+\frac{1}{d-1}\varphi(\overline{y}),$ where $\overline{y}$ is the \textbf{parent} of $y$, i.e. its unique neighbour on the shortest path from $\circ$ to $y$.
Then 
\[\varphi \overset{d.}{=}\phid.
\]
\end{proposition}
\noindent
Let $\eta(h):=\dP^{\Td}(\Ch\text{ is infinite})$. 
\\
Proposition~\ref{prop:recursivegfftrees} is the corollary of a more general domain Markov property (see \cite{ACregultrees}, (1.7)-(1.9) for proof details). Namely, for $U\subsetneq \Td$, define the Green function $\gtd^U$ of the random walked killed when exiting $U$ by 
\[
\gtd^U(x,y)= \dE_x\left[\sum_{k=0}^{T_U}\mathbf{1}_{X_k}=y\right],
\]
where $T_U:=\inf\{k\geq 0, X_k \not\in U\}$. Define the field $\phid^U$ on $\Td$ by $\phid^U(x)=\phid(x)-\dE\left[\phid(X_{T_U})\right]$ for all $x\in \Td$.

\begin{proposition}[\textbf{Domain Markov property}]\label{prop:domainMarkov}
$\phid^U$ is a Gaussian process with covariances given by $\Cov(\phid^U(x), \phid^U(y))=\gtd^U(x,y),$ and it is independent of $(\phid(x))_{x\not\in U} $.
\end{proposition} 
\noindent
For $k\geq 1$, let $\cZ_k^{h,+}:=\Chplus \cap \partial B_{\Td^+}(\circ,k)$. Define $\cZ_k^h:=\Ch\cap \partial B_{\Td}(\circ,k)$.
\\
Let $\nu:= \cN(0,\frac{d-1}{d-2})$, $\nu_1 := \cN(0,\frac{d}{d-1})$, and $L^2(\nu):=L^2(\dR, \mathcal{B}(\dR),\nu)$. For $h\in \dR$, define the operator $L_h$ on $L^2(\nu)$ by
\begin{equation}\label{eqn:Lhdef}
(L_hf)(a):=(d-1)\mathbf{1}_{[h,+\infty)}(a) \dE_Y\left[f\left(\frac{a}{d-1}+Y\right)\mathbf{1}_{[h,+\infty)}\left(\frac{a}{d-1}+Y\right)\right]
\end{equation}
for all $f\in L^2(\nu)$ and $a\in \dR$, where $Y\sim \nu_1$ and $\dE_Y$ is the expectation w.r.t. $Y$. 

\noindent
By Proposition~\ref{prop:recursivegfftrees}, one has $(L_hf)(a)=\dE_a^{\Td}[\sum_{x\in \cZ_1^{h,+}}f(\phid(x))]$, where $\dE_a^{\Td}$ is the expectation conditionally on $\phid(\circ)=a$. By a straightforward induction, for all $k\geq 1$, the $k$-th iterate of $L_h$ is given by
\begin{equation}\label{eqn:operatorrecursiveequation}
(L_h^kf)(a):=\dE_a^{\Td}\left[\sum_{x\in \cZ_k^{h,+}}f(\phid(x))\right].
\end{equation}
Informally, $L_h$ encodes how the information travels from one generation to the next in~$\Chplus$. 

\begin{proposition}[Propositions 3.1 and 3.3 of \cite{SZ2016}, Proposition 2.1 of~\cite{ACregultrees}]\label{prop:SZ16}
$L_h$ is a self-adjoint and non-negative operator, its norm $\lambda_h$ corresponds to a simple eigenvalue. $h\mapsto\lambda_h$ is a decreasing homeomorphism from $\dR$ to $(0, d-1)$, and $h_{\star}$ is the unique value such that $\lambda_{h_{\star}}=1$.
\\
Let $\chi_h$ be the corresponding eigenfunction such that $\Vert \chi_h\Vert_{L^2(\nu)}=1$: it vanishes on $(-\infty, h)$ and it is continuous and positive on $[h, +\infty)$. 
\end{proposition}
%

\noindent
The construction of Proposition~\ref{prop:recursivegfftrees} gives a monotonicity property for the GFF on $\Td$. A set $S\subset \dR^{\Td}$ is said to be \textbf{increasing} if for any $(\Phi^{(1)}_z)_{z\in \Td}, (\Phi^{(2)}_z)_{z\in \Td}\in S$ such that $\Phi^{(1)}_z \leq \Phi^{(2)}_z$ for all $z\in \Td$, $  (\Phi^{(1)}_z)_{z\in \Td}\in S\Rightarrow (\Phi^{(2)}_z)_{z\in \Td}\in S$. Say that an event of the form $\{\phid \in S\}$ is \textbf{increasing} if $S$ is increasing.

\begin{lemma}[\textbf{Conditional monotonicity}]\label{lem:monotonicityphid}
If $E$ is an increasing event, then the map $a \mapsto \dP_a^{\Td}(E)$ is non-decreasing on $\dR$.
\end{lemma}

\begin{proof}
Let $a_1,a_2\in \dR$ such that $a_1>a_2$. It suffices to give a coupling between a GFF $\phid^{(1)}$ conditionned on $\phid^{(1)}(\circ)=a_1$ and a GFF $\phid^{(2)}$ conditionned on $\phid^{(1)}(\circ)=a_1$ such that a.s., for every $z\in \Td$, $\phid^{(1)}(z)\geq \phid^{(2)}(z)$. To do this, let $(\zeta_y)_{y\in \Td}$ be i.i.d. standard normal variables, and define recursively $\phid^{(1)}$ and $\phid^{(2)}$ as in Proposition~\ref{prop:recursivegfftrees}.
\end{proof}

\noindent
We will need another operator: we define $R_h :L^2(\nu) \rightarrow L^2(\nu)$ by 
\begin{equation}\label{eqn:Rhdef}
R_hf(a):=\mathbf{1}_{(-\infty,h)}(a)+\mathbf{1}_{[h,+\infty)}(a) \dE_Y\left[f\left(\frac{a}{d-1}+Y\right)\right]^{d-1}
\end{equation}
for every $f\in L^2(\nu)$ and $a\in \dR$. We refer the reader to Section 3 of \cite{ACregultrees} for details. Note that $R_hf(a)=\dE^{\Td}_a[\prod_{y\in \cZ_1^{h,+}}f(\phid(y)) ]$, and that by a straightforward induction on $k\geq 1$, 
\begin{equation}\label{eqn:Rhnintermsofchildren}
R_h^kf(a):= \dE^{\Td}_a\left[\prod_{y\in \cZ_k^{h,+}}f(\phid(y)) \right],
\end{equation} 
where $R_h^k$ is $R_h$ iterated $k$ times.

\begin{lemma}[Lemma 3.5 in~\cite{ACregultrees}]\label{lem:fixedpointRh}
 $q_h$ and $\mathbf{1}_{(-\infty, +\infty)}$ are the only fixed points of $R_h$ in $\cS_h:=\{f\in L^2(\nu)\, \vert \, 0\leq f\leq 1\text{ and }f=1\text{ on }(-\infty,h)\}$, where for all $a\in \dR$, 
 \begin{equation}\label{eqn:qhdef}
 q_h(a):=\dP_a^{\Td}(\vert\Chplus \vert =+\infty).
 \end{equation}
\end{lemma}
%

\subsection{Exponential growth}\label{subsec:expogrowthCh}
\noindent
We list below some quantitative estimates from~\cite{GckGffPublished} (Propositions~3.4, 3.6 and Corollary~3.5) on the exponential growth of $\vert\cZ_k^h\vert$. All these results hold when replacing $\Ch$ by $\Chplus$, and $\cZ_k^h$ by $\cZ_k^{h,+}$. 
%

\noindent
There are upper and lower large deviations for the growth rate of $\cZ_k^h$:
\begin{proposition}\label{prop:Chlargedevgrowthrate}
For every $\varepsilon >0$, there exists $C> 0$ such that for every $k\in \dN$ large enough, 
\begin{equation}\label{eqn:expomomentssize}
\max_{a\geq h}\dP^{\Td}_a(k^{-1}\log\vert \cZ_k^h\vert \not\in (\log(\lambda_h-\varepsilon), \log(\lambda_h +\varepsilon)+k^{-1}\log\chi_h(a))\,\vert \,\cZ_k^h \neq \emptyset)\leq \exp(-Ck). 
\end{equation}
\end{proposition}

\noindent
In addition, $\vert \Ch\vert$ has exponential moments: 
\begin{proposition}\label{prop:expomomentsCh}
Fix $h<h_{\star}$. There exists a constant $\Cv >0$ such that as $k\rightarrow +\infty$,
\begin{equation}\label{eqn:expomoments}
\max_{a\geq h}\dP^{\Td}_a(k\leq \vert \Ch \vert <+\infty)=o(\exp(-\Cv k)).
\end{equation}
\end{proposition}

\noindent
Noticing that $\{\cZ_k^h\neq \emptyset\} \subset \{\vert \Ch\vert \geq k\} $, we have the following straightforward consequence:
\begin{corollary}\label{cor:expomomentsheight}
For $k$ large enough, for every $a\geq h$,\\
$\dP^{\Td}_a(\Ch\text{ is infinite})-e^{-\Cv k}\leq \dP^{\Td}_a(\cZ_k^h\neq \emptyset)\leq \dP^{\Td}_a(\Ch\text{ is infinite})$.
\end{corollary}

%



\section{Transience}\label{sec:transience}
\noindent
The aim of this Section is to prove Proposition~\ref{prop:transienceintro}. The first step is to prove that $\Chplus$ is a.s.~transient, conditionally on being infinite (Lemma~\ref{lem:transience} below).
\\
For an infinite tree $T$ with root $r$, a \textbf{cutset} $\Pi$ is a finite set of vertices of $T\setminus\{r\}$ such that no vertex of $\Pi$ is in the offspring of another, and such that for every vertex $z\in T\setminus\Pi$, either $z$ is in the offspring of a vertex of $\Pi$, or $\cO_z\setminus \{\cup_{z'\in \Pi}\cO_{z'} \}$ is finite. If $(\Pi_n)_{n\geq 0}$ is sequence of cutsets, say that $\Pi_n\rightarrow \infty$ if $\min_{z\in \Pi_n}\h_T(z)\rightarrow +\infty$ as $n\rightarrow +\infty$. Define the \textbf{branching number} of $T$ as 
$\text{br}(T):=\inf\{\lambda >0,\, \inf_{\Pi}\sum_{z\in\Pi}\lambda^{-\h_T(z)}=0 \}$.
\\
For $p>0$, let $T_p$ be the random tree obtained from $T$ by edge percolation with probability $p$: one suppresses each edge of $T$ with probability $1-p$, independently of the other edges. Let $T_p(r)$ be the connected component of $r$ in $T_p$. The \textbf{critical percolation threshold} of $T$ is defined as $p_c(T):=\inf\{p\geq 0,\, \dP(\vert T_p(r)\vert =+\infty)>0\}$. By Theorem 6.2 of \cite{Lyons90}, $p_c(T)^{-1}=\text{br}(T)$.

\begin{lemma}[Transience of the SRW]\label{lem:transience}
For almost every infinite realization of $\Chplus$, we have
\begin{equation}\label{eqn:branchingnumber}
p_c(\Chplus)^{-1}=\emph{br}(\Chplus)=\lambda_h,
\end{equation}
and the SRW is transient.
\end{lemma}

\begin{proof}
Theorem 4.3 of \cite{Lyons90} states that if $\text{br}(\Chplus)>1$, then the SRW is transient on $\Chplus$, so that we only have to show \eqref{eqn:branchingnumber}.
\\
 \textbf{Upper bound.} We first show that $\text{br}(\Chplus)\leq \lambda_h$. Let $\varepsilon>0$. Note that $\cZ_k^{h,+}$ is a cutset of $\Chplus$. Let $\varepsilon>0$. By Proposition~\ref{prop:Chlargedevgrowthrate} and Corollary~\ref{cor:expomomentsheight} (recall that these statements hold also for $\Chplus$ and $\cZ_k^{h,+}$), there exists a constant $c>0$ such that 
 \[
 \limsup_{k\rightarrow +\infty} e^{ck}\,\dP^{\Td}(\vert\cZ_k^{h,+}\vert \geq (\lambda_h+\varepsilon/2)^k \, \vert\, \vert \Chplus\vert = +\infty)\leq 1.
 \]
 By a union bound on $j\geq k$, we have
 \[
  \limsup_{k\rightarrow +\infty} e^{ck}\,\dP^{\Td}(\exists j \geq k,\, \vert\cZ_j^{h,+}\vert \geq (\lambda_h+\varepsilon/2)^j \, \vert\, \vert \Chplus\vert = +\infty)\leq (1-e^{-c})^{-1}.
 \]
Thus, on $\{\vert \Chplus\vert = +\infty\}$, there exists a.s.~a (random) integer $k_0\geq 1$ such that for every $k\geq k_0$, $\vert\cZ_k^{h,+}\vert \leq (\lambda_h+\varepsilon/2)^k$. This ensures that $\liminf_{k\rightarrow \infty}\sum_{z\in\cZ_k^{h,+}}(\lambda_h+\varepsilon)^{-\h_T(z)}=0 $, so that $\text{br}(\Chplus)\leq \lambda_h+\varepsilon$. 
\\
\textbf{Lower bound.} Reciprocally, for $p\in (0,1)$, let $\Chplusp$ be the connected component of $\circ$ of $\Chplus$ after edge percolation with probability $p$ on $\Td$ (perform this percolation independently of $\phid$). Write $\dP^{\Td,p}$ for the corresponding probability, and $\dP^{\Chplus}$ for $\dP^{\Td,p}$ conditionally on the realization of $\Chplus$. For $a\in \dR$, let $\cP_p(a):=\dP^{\Td}_a(\dP^{\Chplus}(\vert \Chplusp \vert <+\infty) =1)$. Clearly, for every tree $T$ with root $r$, $ \dP(\vert T_p(r)\vert <+\infty)=1$ if and only if $\dP(\vert T^{(i)}_p(r)\vert <+\infty)=1$ for every $i$, where the $T^{(i)}$'s are the subtrees of the children of $r$. Therefore, $\cP_p(a)=R_h\cP_p(a)$ (recall the definition of $R_h$ in (\ref{eqn:Rhdef})). This implies that $\cP_p\in \cS_h$, and by Lemma~\ref{lem:fixedpointRh}, either $\cP_p=q_h$ or $\cP_p=\mathbf{1}_{(-\infty,+\infty)}$.
\\
Take $p>1/\lambda_h$. Then $L_h^{(p)}:=pL_h$ has a largest eigenvalue $p\lambda_h>1$ and  $\chi_h$ is the corresponding normalized eigenfunction. Using this, one might readily adapt the proof of Proposition~3.3 of \cite{SZ2016} to see that $\dP^{\Td,p}(\vert \Chplusp\vert =+\infty)>0$. Since $\dP^{\Td,p}(\vert \Chplusp\vert =+\infty)=\int_{\dR}(1-\cP_p(a)) \nu(da)$, this forces $\cP_p=q_h$. 
\\
Therefore, we have 
$$\dP^{\Td}_a(\vert\Chplus\vert <+\infty)=q_h(a)=\cP_p(a)=\dP^{\Td}_a(\dP^{\Chplus}(\vert \Chplusp \vert <+\infty) =1)$$ 
for every $a\in \dR$. Integrating over $a\geq h$, we obtain that 
$$\dP^{\Td}(\vert\Chplus\vert <+\infty) =\dP^{\Td}(\dP^{\Chplus}(\vert \Chplusp \vert <+\infty) =1).$$
Since $\{ \vert\Chplus\vert <+\infty\}\subset \{\dP^{\Chplus}(\vert \Chplusp \vert <+\infty) =1 \}$, we deduce that 
$$\dP^{\Td}(\{ \dP^{\Chplus}(\vert \Chplusp \vert <+\infty) =1\} \cap \{\vert\Chplus\vert =+\infty \})=0.$$ 
Since the conditioning on $\{\vert \Ch\vert =+\infty \}$ is non-degenerate under $\dP^{\Td}$, it follows that for almost every realization of $\Chplus$ such that $\vert \Chplus\vert = +\infty$, we have $\dP^{\Chplus}(\vert \Chplusp \vert =+\infty) >0$ and thus $p_c(\Chplus)\leq 1/\lambda_h$. This concludes the proof.
\end{proof}%

%

\begin{proof}[Proof of Proposition~\ref{prop:transienceintro}]
Note that for every tree $T$ such that the SRW is transient, there exists $\delta(T)>0$ such that the SRW is $\delta(T)$-transient. The map $\delta \mapsto \dP^{\Td}(\Chplus\text{ is $\delta$-transient})$ is non-increasing and $\lim_{\delta \rightarrow 0} \dP^{\Td}(\Chplus\text{ is $\delta$-transient})= \dP^{\Td}(\Chplus\text{ is transient})>0$. Hence, there exists $\delta_0>0$ small enough such that 
$$\int_{a\in \dR} q_{h,2d\delta_0}(a)\nu(da)=\dP^{\Td}(\Chplus\text{ is not $2d\delta_0$-transient})<1.$$
\noindent
The event $\{\Chplus\text{ is $2d\delta_0$-transient}\}$ is increasing, thus by Lemma~\ref{lem:monotonicityphid}, $a\mapsto q_{h,2d\delta_0}(a)$ is non-increasing. Hence, for some $a_1$ large enough, $q_{h,2d\delta_0}(a)\leq q_{h,2d\delta_0}(a_1)<1$ for all $a\geq a_1$. Now, there exists $\delta_{a_1}>0$ such that for all $a\geq h$, 
$$\dP^{\Td}_a(\text{$\circ$ has one child $z\in \Chplus$ such that $\phid(z)>a_1$})>\delta_{a_1}.$$

\noindent
By Proposition~\ref{prop:domainMarkov} the subtree  $\cT_z$ from $z$ in $\Chplus$ is $2d\delta_0$-transient with probability at least $1-q_{h,2d\delta_0}(a_1)$. In this case, $\Chplus$ is $\delta_0$-transient (if a SRW starts from $\circ$, it goes to $z$ with probability at least $1/d$, makes its next move in $\cT_z$ with probability at least $1/2$, and then has probability at least $2d\delta_0$ to stay forever in $\cT_z$).
\\
Therefore, for every $a\geq h$, $q_{h,\delta_0}(a)< 1-\delta_{a_1}(1-q_{h,2d\delta_0}(a_1))$. Since $\delta\mapsto q_{h,\delta}(a)$ is non-decreasing for every fixed $a$, this concludes the proof with $\varepsilon = \delta_{a_1}(1-q_{h,2d\delta_0}(a_1))$.
\end{proof}

\section{Renewal}\label{sec:renewal}


%

\noindent 
In this section, we show Proposition~\ref{lem:expomomentsdiscovery}.  The structure of the proof is similar to that of the analogous result for Galton-Watson trees, namely Theorem~2 of Piau~\cite{Piau}. There are nonetheless several changes due to the dependencies induced by the GFF, and we could only find a French version of~\cite{Piau}, so that we give a full proof. 
As mentioned in Section~\ref{subsec:proofstrategy}, a central tool is the following Lemma, which ensures that $\Chplus$, when infinite, has on any of its finite paths from the root a linear number of escape ways to infinity for the SRW. We postpone its proof to the Appendix~\ref{subsec:appendixgrimmettkesten}. 
\\
For a rooted tree $T$, for $y\in T$ and $z\in\xi_y\setminus\{y\}$ (recall that $\xi_y$ is the shortest path from $\circ$ to $y$), say that $z$ is a \textbf{$\delta$-exit} if $z$ has a child $z'\not\in \xi_y$ such that the subtree from $z'$ in $T$ is $\delta$-transient. For $z\in \Chplus$, denote $E(z,\delta)$ the number of $\delta$-exits on $\xi_z$.

\begin{lemma}\label{lem:grimmettkesten}
There exist constants $\delta_1,\Cvi , \Cvii , \Cviii >0$ such that for every $k\geq 1$ and $a\geq h$,
$$\dP^{\Td}_a\left(\min_{z \in B_{\Ch}(\circ,k)} E(z,\delta_1) \leq \Cvi k \right) \leq \Cvii e^{-\Cviii  k}.$$  
\end{lemma}

\noindent
We prove Proposition~\ref{lem:expomomentsdiscovery} by an annealed exploration of $\Chplus$ and the SRW $(X_k)_{k\geq 0}$ on it, by revealing the vertices of $\Chplus$ when $(X_k)$ visits them. We decompose the trajectory of $(X_k)$ into excursions between new height records. In the first step of the proof, we use the fact that each time $(X_k)$ reaches such a record at some vertex $x$, there is a probability bounded away from $0$ that $\cT_x$, the subtree from $x$ in $\Chplus$, is $\delta_0$-transient by Proposition~\ref{prop:transienceintro}. In the second step, Lemma~\ref{lem:grimmettkesten} helps to ensure that the distance between consecutive records has exponential moments. In the third step, we make sure that the random walk does not lose too much time in the finite bushes of $\Chplus$ (i.e.~the subtrees $\cT_x$ such that $\cT_x$ is finite, for $x\in \Chplus$). 

\begin{proof}[Proof of Proposition~\ref{lem:expomomentsdiscovery}]
Fix $a\geq h$. We decompose the trajectory of $(X_k)$ on $[0, \tau_1]$ as follows: if $X_k\neq \circ$ for all $k\geq 1$, $\tau_1=1$. Else, let $r_1:=\inf\{k\geq 1, \vert X_k\vert =0 \}$ be the time of the first return to the root, $m_1:=\sup\{\vert X_k\vert ,\, k\leq  r_1\}$ the largest height of the trajectory during this excursion, and $s_1:=\inf\{k\geq 1, \, \vert X_k\vert =1+m_1\}$ the first time that the walk reaches a higher point. 
\\
For $i\geq 1$, if $r_i,m_i$ and $s_i$ have been defined with $r_i<+\infty$, then set $r_{i+1}:=\inf\{k\geq s_i, \, \vert X_k\vert =\vert X_{s_i}\vert -1\}$ the first return below the record $\vert X_{s_i}\vert $, $m_{i+1}:=\sup\{\vert X_k\vert -\vert X_{s_i}\vert +1,\, k\leq  r_{i+1}\}$ the height of the excursion between $s_i$ and $r_{i+1}$ and $s_{i+1}:=\inf\{k\geq 1, \, \vert X_k\vert =\vert X_{s_i}\vert +m_{i+1}\}$ the first time that the walk reaches a new record after that excursion. If $(\vert X_k\vert )_{k\geq s_i}$ stays forever above $m_i$ (i.e.~$r_{i+1}=+\infty$), then $\tau_1=s_i$. Let $i_0:=\inf\{i\geq 1, \, s_i= \tau_1\}$.  
\\ 
\\
\textbf{Step 1:}
We claim that there exists $\epsilon >0$ (only depending on $d$ and $h$) such that for every $i\geq 1$, 
\begin{equation}\label{eqn:step1}
\sup_{a\geq h}\sP^{h,+,\infty}_a(i_0\geq i+1 \vert i_0\geq i)<1-\epsilon.
\end{equation}
For every $i,\ell\geq 1$, every rooted tree $T$ of height $\ell$, every vertex $y\in T$ of height $\ell$ and $b\geq h$, we have 
\begin{align*}
p:&=\sP^{h,+,\infty}_a(i_0=i\vert\,  B_{\Chplus}(\circ,\ell )=T,\, X_{s_i}=y,\, \phid(y)=b)
\\
&=\sP^{h,+,\infty}_a(\forall k\geq s_i, X_k \in \cT_y\vert\,  B_{\Chplus}(\circ,\ell )=T,\, X_{s_i}=y,\, \phid(y)=b)
\\
&\geq\delta\, \sP^{h,+,\infty}_a(\cT_y\text{ is $\delta$-transient}\vert\,  B_{\Chplus}(\circ,\ell )=T,\, X_{s_i}=y,\, \phid(y)=b),
\end{align*}
where $\cT_y$ is the subtree in $\Chplus$ from $y$. The second line follows from the strong Markov property for the SRW (as $s_i$ is a stopping time for $(X_k)$ w.r.t. to its quenched filtration, i.e. the canonical filtration of $(X_k)$ conditionally on the realization of $\phid$, and hence of $\Chplus$). The third line follows from Markov's inequality. Denote $\dP^{\Td}_{b,y}$ the law of $\phid$ conditionally on $\phid(y)=b$: we have 
\begin{align*}
p&\geq \delta\, \sP^{h,+}_a(\cT_{y}\text{ is $\delta$-transient} \vert\, B_{\Chplus}(\circ,\ell )=T,\,\, \phid(y)=b,\,\vert \Chplus\vert = +\infty)
\\
&\geq \delta\, \dP^{\Td}_{b,y}(\cT_{y}\text{ is $\delta$-transient} \vert\, B_{\Chplus}(\circ,\ell )=T,\,\vert \Chplus\vert = +\infty)
\\
&\geq   \delta\, \dP^{\Td}_{b,y}(\cT_{y}\text{ is $\delta$-transient} \vert\, B_{\Chplus}(\circ,\ell )=T)
\end{align*}
since  $\{\cT_{y}\text{ is $\delta$-transient}\} \cap  \{B_{\Chplus}(\circ,\ell )=T\}\subset \{\vert \Chplus\vert = +\infty)\} \cap \{B_{\Chplus}(\circ,\ell )=T\}$. Now, by Proposition~\ref{prop:domainMarkov}, conditionally on $\phid(y)$, $\{\cT_{y}\text{ is $\delta$-transient}\} $ and $ \{B_{\Chplus}(\circ,\ell )=T\} $ are independent. Hence taking $\delta =\delta_0/2$, we have
\[
p \geq  \delta\, \dP^{\Td}_{b,y}(\cT_{y}\text{ is $\delta$-transient})=\delta \dP^{\Td}_b(\Chplus\text{ is $\delta$-transient})\geq \delta \varepsilon
\] 
for some $\varepsilon>0$ depending on $\delta_0$ (which is itself a function of $d$ and $h$) by Proposition~\ref{prop:transienceintro}. Taking $\epsilon = \delta\varepsilon $ yields \eqref{eqn:step1}.
\\
\\
\textbf{Step 2:} We establish the existence of $\Cix  ,\Cx  >0$  such that for every $k\geq 1$,
\begin{equation}\label{eqn:step2}
\sup_{a\geq h}\sP^{h,+,\infty}_a(\vert X_{\tau_1}\vert \geq k) \leq \Cix  e^{-\Cx  k}.
\end{equation} 
Note that $\vert X_{\tau_1}\vert =m_1+\ldots +m_{i_0}+1$.
We start by showing that $m_i$ has exponential moments, uniformly in $i\geq 1$ and $a\geq h$, by applying Lemma~\ref{lem:grimmettkesten} at the subtree rooted at $X_{s_i}$. We then combine this with a bound on $i_0$ derived from Step 1.
\\
For $k\geq 0$, denote $\cT_k $ the subtree from $X_k$ in $\Chplus$. 
Remark that for all $x\in \Td^+$, conditionally on $\phid(x)$, $\phid$ on the subtree from $x$ in $\Td^+$ is distributed as $\phid$ on $\Td^+$ under $\dP^{\Td}_{\phid(x)}$. By Proposition~\ref{prop:domainMarkov}, we then have for $i\geq 1$ and for every set $\cA$ of rooted trees:
\begin{equation}\label{eqn:Tsievent}
\sP^{h,+}_a(\cT_{{s_i}}\in \cA)=\sum_{x\in \Td^+}\sP^{h,+}_a(X_{s_i}=x)\dE_{Z_{a,x}}[\dP^{\Td}_{Z_{a,x}}(\Chplus \in \cA)]\leq \sup_{b\geq h}\dP^{\Td}_b(\Chplus \in \cA),
\end{equation}
where $Z_{a,x}$ has the distribution of $\phid(x)$ under $\sP^{h,+}_a(\,\cdot\,\vert X_{s_i}=x) $ and $\dE_{Z_{a,x}}$ is the associated expectation. Since $q_h$ is non-increasing, we have for any event $\cE$ and $a\geq h$:
\begin{equation}\label{eqn:absolcontsPa}
\sP^{h,+,\infty}_a(\cE)\leq (1-q_h(a))^{-1}\sP^{h,+}_a(\cE)\leq (1-q_h(h))^{-1}\sP^{h,+}_a(\cE).
\end{equation}
Combining this with \eqref{eqn:Tsievent} and Lemma~\ref{lem:grimmettkesten}, with $\cE=\{\min_{z \in B_{\cT_{s_i}}(X_{s_i},k)} E(z,\delta_1) \leq \Cvi k\}$ and $\cA$ the set of rooted trees $T$ such that $\cE$ holds for $T=\cT_{s_i}$, we get
\[
\sP^{h,+,\infty}_a\left(\min_{z \in B_{\cT_{s_i}}(X_{s_i},k)} E(z,\delta_1) \leq \Cvi k \right) \leq \Cvii (1-q_h(h))^{-1} e^{-\Cviii k}.
\]
If $\cT_{s_i}$ satisfies $\min_{z \in B_{\cT_{s_i}}(X_{s_i},k)} E(z,\delta_1) \geq \Cvi k$, and if $X_n\in \cT_{s_i}\setminus  B_{\cT_{s_i}}(X_{s_i},k)$ for some $n \in [s_i, r_{i+1}]$, then with probability at least $1- (1-\delta_1)^{\Cvi k}$, $(X_j)_{j\geq n}$ never comes back to $X_{s_i}$, and $i_0=i$. Hence,
\begin{center}
 $\sP^{h,+,\infty}_a(m_i\geq k)\leq \Cvii (1-q_h(h))^{-1} e^{-\Cviii k} + (1-\delta_0)^{\Cvi k}$ for $i,k\geq 1$.
\end{center}
Remark that these bounds are uniform in $a$ and in the value of $\phid(X_{s_i})$. Moreover, conditionally on the value of $\phid(X_{s_i})$, $m_{i+1}$ is independent of $\{m_1, \ldots, m_i\}$. Therefore, under $\sP^{h,+,\infty}_a$, $m_1+\ldots +m_k$ is stochastically dominated by the sum of $k$ i.i.d. variables of some law $\mu$ such that if $Y\sim \mu$, $\dP(Y\geq j) \leq ce^{-c'j}$ for some positive constants $c,c'$ (independent of $a$) and every $j\geq 1$.
\\
Let $K\in (0, \dE[Y]^{-1})$. Then 
\begin{center}
$\sP^{h,+,\infty}_a(\vert X_{\tau_1}\vert \geq k) \leq \sP^{h,+,\infty}_a(i_0 \geq Kk) + \dP(Y_1+\ldots Y_{Kk}\geq k)$, 
\end{center}
where the $Y_i$'s are i.i.d. copies of $Y$. By \eqref{eqn:step1} for the first term of the RHS and the Chernov bound for the second term, if $c$ is large enough and $c'$ small enough, then for every $k\geq 1$ and $a\geq h$,
\[
\sP^{h,+,\infty}_a(\vert X_{\tau_1}\vert \geq k) \leq ce^{-c'k},
\]
and \eqref{eqn:step2} follows.

\noindent
\textbf{Step 3:} The goal of this step is to give a lower bound on the maximal height reached by $(X_n)_{n\leq k}$, showing that it does not lose to much time in finite subtrees of $\Chplus$. Precisely, we establish that if $\Cix  $ is large enough and $\Cx  $ small enough, then for all $a\geq h$ and $k\geq 1$,
\begin{equation}\label{eqn:step3}
\sP^{h,+,\infty}_a\left(\max_{n\leq k}\vert X_n\vert \geq k^{1/6}\right)=\sP^{h,+,\infty}_a(t_k\leq k)\geq 1-\Cix  e^{-\Cx  k^{1/6}},
\end{equation}

\noindent
where $t_k:=\min\{n\geq 1,\,\vert X_n\vert \geq k^{1/6}\}$. Note that $t_k$ is a stopping time. 
We decompose $\Chplus$ as a \textbf{skeleton} $\cS^{h,+}$, the subtree whose vertices are exactly those with an infinite offpsring in $\Chplus$, i.e.~its vertex set is $\{x\in \Chplus, \vert\cO_x \vert=+\infty\}$, to which are attached finite subtrees, called \textbf{bushes}. For instance, if $\Chplus$ is finite, $\cS^{h,+}=\emptyset$ and $\Chplus$ is one single bush. Then, one can decompose the trajectory of $(X_k)$ as a SRW on $\cS^{h,+}$, with excursions in the bushes. 
\\
We prove two things: first, if $(X^{\cS}_n)_{n\geq 0}$ is a SRW on $\cS^{h,+}$, then for some $c$ small enough and every $k$ large enough,
\begin{equation}\label{eqn:symmetricRWonN}
\sP^{h,+,\infty}_a\left(\max_{n\leq k^{1/2}} \vert X^{\cS}_n\vert  \leq  k^{1/6} \right) \leq e^{- ck^{1/6}}.
\end{equation}

\noindent
Second, we control the time lost by the SRW on $\Chplus$ in the bushes, by showing that for $k$ large enough,
\begin{equation}\label{eqn:bushesnotwastetime}
\sP^{h,+,\infty}_a((X_n)_{n\leq t_k}\text{ makes $\lfloor k^{1/2}/2\rfloor$ consecutive steps not on the edges of $\cS^{h,+}$}) \leq e^{- ck^{1/6}}.
\end{equation}
Suppose that these two estimates hold. If $t_k\geq k$, either $(X_n)_{n\leq t_k}$ makes at some point $k^{1/2}/2$ consecutive steps not on the edges of $\cS^{h,+}$, or $(X_n)_{n\leq t_k}$ makes at least $k^{1/2}$ steps on the edges on $\cS^{h,+}$ (not necessarily consecutive). Note that the trace of $(X_n)_{n\geq 0}$ on $\cS^{h,+}$ is distributed as $(X^{\cS}_n)_{n\geq 0}$. Thus, by \eqref{eqn:symmetricRWonN} and \eqref{eqn:bushesnotwastetime} respectively, each of these two alternatives has probability at most $e^{- ck^{1/6}}$. Therefore, for $k$ large enough,
\[
\sP^{h,+,\infty}_a(t_k\leq k)\leq 2e^{- ck^{1/6}},
\]
and \eqref{eqn:step3} follows. Hence, we are left with showing \eqref{eqn:symmetricRWonN} and \eqref{eqn:bushesnotwastetime}.
\\

\noindent
\textbf{Proof of \eqref{eqn:symmetricRWonN}.} Remark that $\Chplus$ is a.s. such that $(\vert X^{\cS}_n\vert )_{n\geq 0}$ dominates stochastically a SRW $(Y_n)_{n\geq 0}$ on $\mathbb{N}_0$ reflected at $0$. By Donsker's theorem, there exists a positive constant $c>0$ such that for $k$ large enough, for every $i\geq 0$, $\dP(Y_{\lfloor k^{1/3}\rfloor }\leq k^{1/6}\vert Y_0=i) \leq e^{-2c}.$ Applying the simple Markov property at times $n\lfloor k^{1/3}\rfloor$ for $n=1, 2, \ldots, \lfloor k^{1/6}\rfloor$ yields (\ref{eqn:symmetricRWonN}).
\\
\\
\textbf{Proof of \eqref{eqn:bushesnotwastetime}.} 
%
On the first $\lfloor k^{1/6}\rfloor $ generations of $\Chplus$, there are less than $d(d-1)^{k^{1/6}}$ vertices. By a union bound on these vertices and Proposition~\ref{prop:expomomentsCh}, if $C'>0$ is large enough, then for $k$ large enough, 
\begin{equation}\label{eqn:probaEnbushes}
\sP^{h,+,\infty}_a(E_k)=\dP^{\Td}_a(E_k)\geq 1-e^{-k^{1/6}},
\end{equation} 
where $E_k=\{\text{the largest bush of }B_{\Chplus}(\circ, \lfloor k^{1/6}\rfloor)\text{ has size at most }C'k^{1/6}\}$. 
\\
Let $B$ be a bush such that $\vert B \vert\leq C'k^{1/6}$, and $x_B$ its root (hence, $x_B$ has exactly one neighbour in $\cS^{h,+}$). By Theorem~1 of \cite{KahnLinialCovertime}, if $k$ is large enough (depending only on $C'$), then the expected hitting time of $x_B$ by a SRW started at an arbitrary vertex in $B$ is less than $\lfloor 2C'k^{1/3}\rfloor-4$. Hence by Markov's inequality, a SRW in $\Chplus$ starting at any vertex of $B$ has a probability at least $1/2$ to hit $x_B$ after at most $\lfloor 4C'n^{1/3}\rfloor-2$ steps. From $x_B$, the probability to reach $\cS^{h,+}$ at the next step and to stay in $\cS^{h,+}$ at the step after is at least $d^{-2}$, so that the probability that a SRW starting in $B$ goes through an edge of $\cS^{h,+}$ after at most $\lfloor 4C'n^{1/3}\rfloor$ steps is at least $d^{-2}/2$.
\\
Fix a realization of $\Chplus$ such that $E_k$ holds. Let $x\in B_{\Chplus}(\circ,\lfloor k^{1/6}\rfloor)\setminus \cS^{h,+}$ (if this set is empty, then a SRW starting at $\circ$ can not make even one step in a bush before $t_k$, so that we can discard this case). Start at SRW at $x$. 
By the Markov property applied at times $\lfloor 4C'k^{1/3}\rfloor i $ for $1\leq i\leq k^{1/2}/( 2\lfloor4C'k^{1/3}\rfloor)$, there exists $C>0$ (only depending on $C'$ and $d$) such that for $k$ large enough (uniformly in the realization of $\Chplus$), the probability that this SRW makes at least $k^{1/2}/2$ consecutive steps without crossing an edge of $\cS^{h,+}$ is less than $e^{-Ck^{1/6}}$. Therefore, by the Markov property again, writing
\[
p_n:=\bP^{\Chplus}_{\circ}(X_n\in B_{\Chplus}(\circ,\lfloor k^{1/6}\rfloor)\setminus \cS^{h,+} ,\, (X_j)_{n\leq j\leq n+\lfloor k^{1/2}/2\rfloor}\,\text{does not cross an edge of}\,\cS^{h,+})
\]
and
\[
F_k:=\{ (X_n)_{n\leq t_k}\text{ makes $\lfloor k^{1/2}/2\rfloor$ consecutive steps not on the edges of $\cS^{h,+}$}\},
\]
we get that
\begin{center}
$\bP^{\Chplus}_{\circ}(F_k)\leq \sum_{n\leq k}p_n \leq ke^{-Ck^{1/6}}.$
\end{center}
Hence, recalling~\eqref{eqn:probaEnbushes} and choosing $c\in (0,C)$, we have
\begin{align*}
&\sP^{h,+,\infty}_a(F_k)  
\leq \sP^{h,+,\infty}_a(E_k^c)+ ke^{-Ck^{1/6}} \leq e^{-k^{1/6}} +ke^{-Ck^{1/6}} \leq e^{-ck^{1/6}}.
\end{align*} 
This shows \eqref{eqn:bushesnotwastetime}, and Step 3 is completed.
\\
\\
\textbf{Conclusion:} We now combine Steps 2 and 3 to finish the proof: by \eqref{eqn:step2} (with $\lfloor k^{1/6}\rfloor$ instead of $k$) and \eqref{eqn:step3}, we have for $k\geq 1$:
\[
\sP^{h,+,\infty}_a(\vert X_{\tau_1}\vert \leq \max_{n\leq k}\vert X_n\vert )\geq 1-2\Cix  \exp^{-\Cx  k^{1/6}/2}.
\]
On this event, $\tau_1\leq k$, so that we can take $\Ci=2\Cix  $ and $\Cii =\Cx  /2$.
\end{proof}


\noindent
Denote $\cT_{\tau_i}$ the subtree from $X_{\tau_i}$ in $\Chplus$. 
\begin{remark}\label{rem:renewalabsolutcont}
The law of $(\cT_{\tau_i},\phid\vert_{\cT_{\tau_i}},(X_k)_{k\geq \tau_i})$ conditionally on $\phid(X_{\tau_i}=a)$ is the law of $(\Chplus,\phid\vert_{\Chplus},(X_k)_{k\geq 0})$ under 
\begin{equation}\label{eqn:defPrenewal}
\sPr_a:=\sP^{h,+,\infty}_a(\,\cdot \,\vert \,\forall k\geq 0,\, X_k\neq \oc )
\end{equation}
for $i\geq 1$. 
Also, by Proposition~\ref{prop:transienceintro}, we have 
\begin{equation}\label{eqn:minprobatransience}
 M_h:=\min_{a\geq h}\sP^{h,+,\infty}_a(\forall k\geq 0,\, X_k\neq \oc)>0
\end{equation}
\noindent
so that the conditioning is uniformly non-degenerate. 
\end{remark}
\noindent
Hence we get the following upgraded version of Proposition~\ref{lem:expomomentsdiscovery}:

\begin{proposition}\label{prop:expomomentsrenewal}
If $\Ci$ is large enough and $\Cii $ small enough, then for every $a,b\geq h$, $i\geq 0$ and $k\geq 1$, 
$$\max(\sP^{h,+,\infty}_a(\tau_{i+1}-\tau_i\geq k\,\vert \, \phid(X_{\tau_i})=b),\,\sPr_a(\tau_{i+1}-\tau_i\geq k\,\vert \, \phid(X_{\tau_i})=b)\,)\leq \Ci e^{-\Cii k^{1/6}}.$$

\noindent
In addition under either $\sP^{h,+,\infty}_a$ or $\sPr_a$, and conditionally on the value of $\phid(X_{\tau_i})$, the triplet $(\cT_{\tau_i},\phid\vert_{\cT_{\tau_i}},(X_k)_{k\geq \tau_i})$ is independent of the triplet $(\Chplus \setminus \cT_{\tau_i} , \phid\vert_{\Chplus\setminus\cT_{\tau_i}}, (X_k)_{0\leq k \leq \tau_i})$.
\end{proposition}

\section{Existence of the speed}\label{sec:speed}
We now turn to the proof of Theorem~\ref{thm:LLN}. We first show Proposition~\ref{prop:CLThtau} in Section~\ref{subsec:regularity}, then a pointwise LLN and CLT (Proposition~\ref{prop:LLNandCLT}) in Section~\ref{subsec:pointwise}, and finally Theorem~\ref{thm:LLN} in Section~\ref{subsec:proofthmmain}. All proofs will be done under $\sPr_h$, as it turns out that they can be adapted to the other annealed laws of interest ($\sP^{h,\infty}$, $\sP^{h,+,\infty}$, $\sP^{h,\infty}_a$ and $\sP^{h,+,\infty}_a$ for $a\geq h$):

\begin{remarks}\label{rem:choiceprobaspace}
We have the following.
\begin{itemize}
\item The proofs of Propositions~\ref{prop:CLThtau} and~\ref{prop:LLNandCLT} as well as Theorem~\ref{thm:LLN} adapt hold under $\sPr_\mu$ for any probability distribution $\mu$ on $[h,+\infty)$, since the results we use from~\cite{MeynTweedie} are valid for any such distribution $\mu$, and since the bound of Proposition~\ref{prop:expomomentsrenewal} are uniform in the value of the GFF at a renewal time.  
\item The first renewal interval is irrelevant: let $\sP$ be an annealed distribution $\sP$ on $\Chplus$ and a SRW on $\Chplus$ such that there exists an a.s.~finite time $\tau\in \dN$ so that $(X_k)_{k\geq \tau}$ is distributed as $(X_k)_{k\geq 0}$ under $\sPr_\mu$ for some ad hoc distribution $\mu$ (that can depend on $\sP$). Applying Theorem~\ref{thm:LLN} to $(\vert X_k\vert -\vert X_{\tau}\vert)_{k\geq \tau}$ and using that $\tau$ (and $\vert X_\tau\vert$) are finite (and hence are tight) yields Theorem~\ref{thm:LLN} for $(X_k)_{k\geq 0}$. We then apply this to $\sP^{h,+,\infty}_a$ for any $a\geq h$, and to $\sP=\sP^{h,+,\infty}$ (since under those distributions, we can take $\tau=\tau_1$).
\item The adaptation from $\Chplus$ to $\Ch$ is immediate once one notices that the proof of Proposition~\ref{prop:transienceintro}  holds for $\Ch$ instead of $\Chplus$, and that the results of Section~\ref{sec:renewal} hold on $\Ch$ as well, as only the law of the first renewal interval changes.
\end{itemize}
\end{remarks}

\subsection{Ergodicity of renewal intervals}\label{subsec:regularity}

\noindent
In this section, we prove Proposition~\ref{prop:CLThtau}. 
%
Recall the definition of $Y_i$ from~\eqref{eqn:defYiMC}. 
By Remark~\ref{rem:renewalabsolutcont}, under $\sPr_h$, $(Y_i)_{i\geq 0}$ is a Markov chain on the state space $\cX:=[h,+\infty)\times \cM$, where $\cM$ is the following countable set. 
\\
For $j\geq 1$, let $\cT_{d}$ be the set of finite rooted trees with each vertex having degree at most $d$. For $T\in \cT_d$, let $\cW_T$ be the set of finite nearest neighbour walks on $T$, starting at the root of $T$ and ending on a vertex of $T$ of maximal height. Let
\begin{equation}\label{eqn:defMset}
\cM:=\{(T,W):\, \,  T\in \cT_d, \, W\in \cW_T\}.
\end{equation}
\noindent
We equip $\cX$ with the sigma-field $\cB(\cX)$ generated by the Borel sets on $[h,+\infty)$ and the power set of $\cM$.
For $Y=(b,T,W)\in \cX$, denote $\varphi(Y):=b$, $\h(Y)$ the height of $T$  and $\tau(Y)$ the length of $W$. Remark~\ref{rem:renewalabsolutcont} also implies that the distribution of $Y_{i+1}$ conditionally on $Y_i$ does not depend on $i\geq 0$.  Let $Q$ be the transition kernel of the Markov chain $(Y_i)_{i\geq 0}$. Moreover, we have that for any $Y\in \cX$, 
\begin{equation}\label{eqn:QYsP+}
\text{the probability measure}\text{ $Q(Y,\cdot)$ is the distribution of $(\phid(X_{\tau_1}), T^{(\tau)}_0,W^{(\tau)}_0  )$ under $\sPr_{\varphi(Y)}$},
\end{equation}
where we recall the definition of $\sPr$ at~\eqref{eqn:defPrenewal}. 
Since this measure only depends on the first coordinate of $Y$, we can define $Q(b,\cdot)$ as $Q((b,T,W),\cdot)$ for any $b\geq h$ and an arbitrary $(T,W)\in \cM$.  Denote $\dP^{Q}$ the probability associated to $(Y_i)_{i\geq 0}$. 
\\
We show that $(Y_i)_{i\geq 0}$ is positive Harris recurrent, and satisfies a drift condition w.r.t.~a potential function that dominates $\h(Y)$ and $\tau(Y)$ for $Y\in \cX$. Theorem 17.0.1 in~\cite{MeynTweedie} then implies Proposition~\ref{prop:CLThtau}.
%

\noindent
We give a short proof of this proposition, which relies on two technical Lemmas that we state and prove below. Lemma~\ref{lem:reachability} essentially gives tightness properties on the sequence $(\varphi(Y_i))_{i\geq 0}$; in particular, it will visit infinitely many times every compact interval of $[h,+\infty)$. Lemma~\ref{lem:Vpotential} states that $(Y_i)_{i\geq 0}$ satisfies a drift condition w.r.t.~the potential function $V$ defined at~\eqref{eqn:Vpotentialdef}.
%

\begin{proof}[Proof of Proposition~\ref{prop:CLThtau}.] 
By Theorem 17.0.1 of~\cite{MeynTweedie}, it is enough to show that $(Y_i)_{i\geq 0}$ is positive Harris recurrent, and $V$-uniformly ergodic to obtain~\eqref{eqn:CLThandtau} and~\eqref{eqn:CLThandtauX}. 
\\
\textbf{Positive Harris recurrence.} By Lemma~\ref{lem:Vpotential} and Theorem~1.2 in~\cite{HairerMattingly}, $(Y_i)_{i\geq 0}$ has a unique invariant measure, that we denote $\pi$. It remains to show that the chain is Harris recurrent. By (9.2) in~\cite{MeynTweedie}, this amounts to show that for a maximal irreducibility measure $\psi$, every $B\subseteq \cX$ such that $\psi(B)>0$ is Harris recurrent, that is 
\begin{equation}\label{eqn:Harrisrecurrentsets}
\dP^Q(\vert\{ i\geq 1,\, Y_i\in B \}\vert =+\infty, \vert\, Y_0=Y)=1
\end{equation}
for every $Y\in B$. We proceed in three steps: first, we find an irreducibility measure $\phi$ for $(Y_i)_{i\geq 0}$, with a finite total mass. Second, we construct a maximal irreducibility measure $\psi$ from $\phi$, via Proposition~4.2.2 of~\cite{MeynTweedie}, and prove that $\psi(B)>0$ only if $B$ contains a subset of $\cX$ of the form $I\times \mathfrak{m}$, where $I$ is a Borel set of $[h,+\infty)$ of positive measure and $\mathfrak{m}\in \cM$. Third, we prove that~\eqref{eqn:Harrisrecurrentsets} holds for any such set $B$.
\\
\textbf{First step.} We define $\phi$ on $\cX$ as follows. Let $\mathfrak{m}_1, \mathfrak{m}_2, \ldots$ be an enumeration of the elements of $\cM$, in an arbitrary order. Let $\phi$ be the unique measure such that for every Borel set $I$ of $[h,+\infty)$ and $j\geq 1$, let $\phi(I\times \mathfrak{m}_j)=2^{-j}\int_I (1+x^2)^{-1}dx$. We now show that this measure is irreducible, that is, for every $B\in \cB(\cX)$ such that $\phi(B)>0$ and every $Y\in \cX$, $\dP^Q(\exists i\geq 1,\,Y_i\in B \,\vert \, Y_1=Y)>0$. 
\\
Let $Y\in \cX$ and let $B$ be such that $\phi(B)>0$. By construction of $\phi$, and since $\cM$ is countable, there exists $j\geq 1$ and $I\subset [h,+\infty)$ of positive Lebesgue measure such that $I\times \{\mathfrak{m}_j\}\subseteq B$. By~\eqref{eqn:reachabilitypositive} and~\eqref{eqn:reachabilitycompact} for an arbitrary $K>0$ and $(T,W)=\mathfrak{m}_j$, we have indeed
\begin{equation}\label{eqn:phiirred}
\dP^Q(\exists i\geq 1,\,Y_i\in B \,\vert \, Y_0=Y)\geq \dP^Q(\exists i\geq 1,\,Y_i\in I\times \{ \mathfrak{m}_j\} \,\vert \, Y_0=Y)>0.
\end{equation}

\noindent
\textbf{Second step.} By Proposition~4.2.2(iv) of ~\cite{MeynTweedie}, since $\phi(\cX)<+\infty$, the measure $\psi$ defined by 
\[
\psi(B)=\sum_{k=0}^{+\infty}2^{-k-1}\int_{\cX}Q^k(Y,B)\phi(dY)
\]
is a maximal irreducibility measure (i.e.~an irreducibility measure such any other irreducibility measure is absolutely continuous w.r.t.~$\psi$). If $I\subseteq [h,+\infty)$ has Lebesgue measure zero, then by~\eqref{eqn:reachabilityzero}, for every $Y\in \cX$, $Q(Y,I\times \cM)=0$. By the chain rule, one extends this easily to $Q^k(Y,I\times \cM)=0$ for all $k\geq 1$. Integrating $Y$ w.r.t.~$\phi$ and summing over $k$, we obtain $\psi(I\times \cM)=0$. 
\\
\textbf{Third step.} Let $B\subseteq \cB(\cX)$ such that $\psi(B)>0$. We have just shown that there must exist $I\subseteq [h,+\infty)$ of positive Lebesgue measure and $\mathfrak{m}\in \cM$ such that $I\times \{\mathfrak{m}\}\subseteq B$. We only have to show that for any $Y\in \cX$, 
\begin{equation}\label{eqn:Harrisrecurrentsetproduct}
\dP^Q(\vert\{ i\geq 1,\, Y_i\in I\times\{\mathfrak{m}\} \}\vert =+\infty, \vert\, Y_0=Y)=1.
\end{equation}
Fix now $Y\in \cX$, and $K=K'>0$. By~\eqref{eqn:reachabilitycompact}, if $Y_0=Y$, there exist $\dP^Q$-a.s.~infinitely many $i$'s such that $\varphi(Y_i)\leq h+K$. By~\eqref{eqn:reachabilitypositive}, there exists $K''>0$ such that for every $i\geq 1$ and $a\in [h,h+K]$, $\dP^Q(Y_{i+1}\in I\times \{\mathfrak{m}\}\,\vert \, \varphi(Y_i) =a)>K''$. Together with the strong Markov property, this establishes~\eqref{eqn:Harrisrecurrentsetproduct}. Hence, we have shown that $(Y_i)_{i\geq 0}$ is positive Harris recurrent.
\\
\\
\textbf{$V$-uniform ergodicity.} We have shown that the chain $(Y_i)_{i\geq 0}$ is $\psi$-irreducible. By~\eqref{eqn:doeblincondition}, $[h,h+\Cxi ]\times \cM$ is a petite set (see §5.5.2 in~\cite{MeynTweedie} for a definition), and by~\eqref{eqn:driftcondition}, the condition (V4) defined at (15.28) holds with $\beta=1/3$, $C=[h,h+\Cxi ]\times \cM$ and $b=\Cxi $. By Theorem 16.0.1(iv) of~\cite{MeynTweedie}, this shows the $V$-uniform ergodicity. This concludes the proof.
\end{proof}

\begin{lemma}\label{lem:reachability}
Let $I$ be a Borel set of $[h,+\infty)$. If $\text{Leb}(I)>0$ where $\text{Leb}$ denotes the Lebesgue measure, then for every $(T,W)\in \cM$ and every $K>0$, 
\begin{equation}\label{eqn:reachabilitypositive}
\inf_{h\leq a\leq h+K}Q(a, (I,T,W))>0.
\end{equation}
Else, if $\text{Leb}(I)=0$, then for all $a\geq h$:
\begin{equation}\label{eqn:reachabilityzero}
Q(a, I\times \cM)=0. 
\end{equation}
Moreover, for all $K'>0$ and $Y\in \cX$, we have
\begin{equation}\label{eqn:reachabilitycompact}
\dP^Q( \vert \{i\geq 1, \, \varphi(Y_i)\leq h+K'  \}\vert =+\infty \,\vert\, Y_0=Y) =1. 
\end{equation}
\end{lemma}

\begin{proof} \textbf{Proof of~\eqref{eqn:reachabilitypositive}.} Let $I\subseteq [h,+\infty)$ such that $\text{Leb}(I)>0$, $(T,W)\in \cM$ and  $K>0$. Note that for all $a\geq h$, we have 
\begin{align*}
Q(x, (I,T,W) )& \geq  \sPr_a(\{\phid(X_{\tau_{1}})\in I\}\cap \{B_{\Chplus}(\circ,\vert X_{\tau_1}\vert ) =T\}\cap \{(X_0,\ldots, X_{\tau_1}) =W\} )
\\
&\geq  \sP^{h,+,\infty}_a(\{\phid(X_{\tau_{1}})\in I\}\cap\{B_{\Chplus}(\circ,\vert X_{\tau_1}\vert ) =T\}\cap \{(X_0,\ldots, X_{\tau_1}) =W\}  ).
\end{align*}
Hence, it is enough to show that
\begin{equation}\label{eqn:reachability}
\inf_{a\in [h,h+K]}\sP^{h,+,\infty}_a(\{\phid(X_{\tau_{1}})\in I\}\cap \{B_{\Chplus}(\circ,\vert X_{\tau_1}\vert  )=T\}\cap \{(X_0,\ldots, X_{\tau_1}) =W\}  ) >0.
\end{equation}
Let $\vert W\vert$ be the length of $W$, and write $W=(x_0, \ldots, x_{\vert W\vert -1})$ with $x_0=\circ$. Let $K'>0$ be such that $\text{Leb}(I\cap [h,K'])>0$. Let $v$ be an arbitrary neighbour of $\circ$, and write $\mathcal{I}=\{[h,h+K],[h-K,h), I\cap [h,K']  \}$ By Proposition~\ref{prop:recursivegfftrees}, we have
\[
r_{K,K'}:=\min_{(I_1,I_2)\in \mathcal{I}^2} \inf_{a\in I_1}\dP^{\Td}_a(\phid(v)\in I_2),
\]
and thus
\begin{equation}\label{eqn:reachabilitypositivetreeT}
\inf_{a\in [h,h+K]} \dP^{\Td}_a(\{\phid(x_{\vert W\vert -1})\in I\}\cap\{B_{\Chplus}(\circ,\vert X_{\tau_1}\vert  =T\})\geq r_{K,K'}^{\vert B_{\Td}(\circ,\h(T))\vert}\geq r_{K,K'}^{d^{\h(T)+1}}.
 \end{equation}
Let $\delta_0$ be as in the proof of Proposition~\ref{prop:transienceintro}. In particular, we have for all $a\geq h$: $q_{h,\delta_0}(a)= \dP^{\Td}_a(\Chplus\text{ is $\delta_0$-transient} )\geq q_{h,\delta_0}(a)$. Letting 
\begin{center}
$\cE:= \{\phid(x_{\vert W\vert -1})\in I\}\cap\{B_{\Chplus}(\circ,\h(T) =T\}\cap \{ \cT_{x_{\vert W\vert -1}}\text{ is $\delta_0$-transient}\}$
\end{center}
where $\cT_{x_{\vert W\vert -1}}$ is the subtree in $\Chplus$ from $x_{\vert W\vert -1}$, we have by~\eqref{eqn:reachabilitypositivetreeT} and Proposition~\ref{prop:domainMarkov}: 
\[
\inf_{a\in [h,h+K]} \dP^{\Td}_a(\cE)\geq  r_{K,K'}^{d^{\h(T)+1}}q_{h,\delta_0}(h).
\]
Finally, we obtain
\begin{align*}
&\inf_{a\in [h,h+K]}\sP^{h,+,\infty}_a(\{\phid(X_{\tau_{1}})\in I\}\cap \{B_{\Chplus}(\circ,\vert X_{\tau_1}\vert  =T\}\cap \{(X_0,\ldots, X_{\tau_1}) =W\}  )
\\
& \geq \dP^{\Td}_a(\cE)\,\times\,\sP^{h,a}(\{(X_0, \ldots, X_{\vert W\vert -1})=W\}\cap\{\forall k\geq \vert W\vert, X_k\in  \cT_{x_{\vert W\vert -1}}\} \,\vert \, \cE)
\\
&\geq  r_{K,K'}^{d^{\h(T)+1}}q_{h,\delta_0}(h) d^{-\vert W\vert}\delta_0 >0. 
\end{align*}
This concludes the proof of~\eqref{eqn:reachabilitypositive}.
\\

\noindent
\textbf{Proof of~\eqref{eqn:reachabilityzero}.} Let $I\subseteq [h,+\infty)$ be such that $\text{Leb}(I)=0$. 
Denoting $v$ an arbitrary neighbour of $\circ$, we have for all $a\geq h$, by Proposition~\ref{prop:recursivegfftrees}: $\dP^{\Td}_a(\phid(v)\in I)=0$. Iterating this to each generation of $\Td$ (whose vertex set is countable), we obtain
\[
\dP^{\Td}_a(\exists y\in \Td\setminus\{\circ\}, \, \phid(y)\in I)=0
\]
for $a\geq h$. Since $\dP^{\Td}_a(\vert \Chplus\vert =+\infty)\geq \dP^{\Td}_h(\vert \Chplus\vert =+\infty)\geq q_{h,\delta_0}(h)>0$, we get
\begin{align*}
\sup_{a\geq h}Q(a,I\times \cM) &\leq \sup_{a\geq h}\sP^{h,+,\infty}_a(\exists y\in \Td\setminus\{\circ\}, \, \phid(y)\in I)
\\
&\leq \sup_{a\geq h} q_{h,\delta_0}(h)^{-1}\dP^{\Td}_a(\exists y\in \Td\setminus\{\circ\}, \, \phid(y)\in I)=0,
\end{align*}
and~\eqref{eqn:reachabilityzero} follows.
\\

\noindent
\textbf{Proof of~\eqref{eqn:reachabilitycompact}.} It is enough to show that if $K>0$ is large enough, then for all $Y\in \cX$, 
\begin{equation}\label{eqn:reachabilitycompactproof}
\dP^Q(\exists i >1, \varphi(Y_i)<h+K\,  \vert \,Y_0=Y)=1.
\end{equation}
Indeed, by the strong Markov property applied to the sequence $(Y_i)_{i\geq 0}$ on the return times of $\varphi(Y_i)$ in $[h,h+K]$, ~\eqref{eqn:reachabilitycompactproof} implies that $\dP^Q$-a.s., there exists an infinite increasing sequence $(i_k)_{k\geq 1}$ sucht that $\varphi(Y_{i_k})\in [h,h+K]$ for every $k$. By~\eqref{eqn:reachabilitypositive} with $I=[h,h+K']\times \cM$, $\min_{k\geq 1}\dP^Q(\varphi(Y_{i_k+1})\in [h,h+K'])>0$. Hence, using again the strong Markov property (as $i_k$ is a stopping time w.r.t.~the canonical filtration of $(Y_i)_{i\geq 0}$), we obtain~\eqref{eqn:reachabilitycompact}. 
\\
We now establish~\eqref{eqn:reachabilitycompactproof}. To do so, we rely on Lemma~\ref{lem:excursionsaboveK} below, which states that for $K$ large enough, if $\phid(X_{\tau_i})>K$, then $\phid(X_{\tau_{i+1}})-\phid(X_{\tau_i})$ has exponential moments and a negative expectation. 
\\
Fix $K > \frac{d-1}{d-2}(1+\dE[\Gamma']+\vert h\vert )>0$, where $\Gamma'$ is defined in Lemma~\ref{lem:excursionsaboveK}. Let $Y\in \cX$, and let $Y_0=Y$. Let $t:=\min\{i\geq 1, \, \varphi(Y_i)<h+K\}$, which is a stopping time w.r.t.~the canonical filtration of $(Y_i)_{i\geq 1}$. We only have to prove that $t$ is $\dP^Q$-a.s.~finite. For every $i\geq 0$, if $\varphi(Y_i)>h+K$, then by~\eqref{eqn:QYsP+} and Lemma~\ref{lem:excursionsaboveK}  the difference $\phid(Y_{i+1})-\phid(Y_{i})$ is stochastically dominated by $W-\frac{d-2}{d-1}(h+K)$, where $W\sim\Gamma'$ (note that for all $a\geq h$, $a-\frac{\max(a,0)}{d-1}\geq \frac{d-2}{d-1}a$). Therefore, for all $m\geq 1$, we have 
%
%
%
\begin{equation}\label{eqn:GFFattauicompact}
\dP^Q(t\geq m \,\vert \,Y_0=Y)\leq \dP\left(\sum_{i=1}^{m-1}\left(W_i -\frac{d-2}{d-1}(h+K)\right) \geq h+K-\varphi(Y) \right),
\end{equation}
\noindent
where the $W_j$'s are i.i.d. variables of law $\Gamma'$. By our choice of $K$, $\dE[W_1 -\frac{d-2}{d-1}(h+K)]<-1$ and by Lemma~\ref{lem:excursionsaboveK}, $W_1$ has exponential moments. Therefore, by the exponential Markov inequality, there exist $c,c'>0$ uniquely depending on $d$ and $h$ so that for every choice of $Y\in \cX$ and every $m\geq 2(\varphi(Y)-h-K)+1 $, 
\begin{equation}\label{eqn:excursionsaboveKbis}
\dP^Q(t\geq m\,\vert \,Y_0=Y)\leq \dP\left(   \sum_{i=1}^{m-1}\left(W_i -\frac{d-2}{d-1}(h+K)\right) \geq (1-m)/2  \right)\leq  ce^{-c'm}.
\end{equation}
By Borel-Cantelli's Lemma (the sequence $(ce^{-c'm})_{m\geq 1}$ being summable), $t$ is a.s.~finite, and this concludes the proof. 
\end{proof}

\begin{lemma}\label{lem:Vpotential}
There exists a large enough constant $\Cxi >0$ such that the following two statements hold. 
\\
1) For every $Y\in \cX$, 
\begin{equation}\label{eqn:driftcondition}
QV(Y)\leq \frac{2V(Y)}{3} +\Cxi \mathbf{1}_{  \{ Y\in [h,h+\Cxi ]\times \cM  \}   }.
\end{equation}
2) There exists a probability measure $\nu_\star$ on $\cX$ and a constant $\alpha >0$ such that for every Borel set $B\in \cB(\cX)$ and every $Y\in [h,h+\Cxi ]\times \cM$,
\begin{equation}\label{eqn:doeblincondition}
Q(Y,B) \geq \alpha \nu_{\star}(B).
\end{equation}
\end{lemma}

\begin{proof} \textbf{Proof of~\eqref{eqn:driftcondition}.} By Proposition~\ref{prop:expomomentsrenewal} and~\eqref{eqn:QYsP+}, there exists $\Cxi >0$ such that
\begin{equation}\label{eqn:Vdominated}
\sup_{Y\in \cX}\int_{Y'\in \cX} (\h(Y')^2+\tau(Y')^2 ) Q(Y,dY') <\frac{\Cxi }{100}.
\end{equation}
Moreover, by Lemma~\ref{lem:excursionsaboveK}, if $\Cxi $ is large enough, then for all $a\geq h+\Cxi $,
\begin{equation}\label{eqn:Vdominatedfirstpart}
\int_{Y'\in \cX} \varphi(Y) Q(a,dY') \leq \dE[W]- \frac{(d-2)a}{d-1}  < \frac{4a}{7},
\end{equation}
where $W\sim \Gamma'$. Hence, taking $\Cxi $ large enough so that~\eqref{eqn:Vdominated} and~\eqref{eqn:Vdominatedfirstpart} hold, we get for every $Y\in \cX\setminus [h,h+\Cxi ]\times \cM$:
\begin{equation}\label{eqn:VdominatedlargeY}
QV(Y)\leq \frac{\Cxi }{100}+ \frac{4\varphi(Y)}{7} \leq \frac{2\varphi(Y)}{3}\leq \frac{2V(Y)}{3}.
\end{equation}
Lemma~\ref{lem:excursionsaboveK} also entails that for $Y\in [h,h+\Cxi ]\times \cM$,  we have 
\[
 \int_{Y'\in \cX} \varphi(Y) Q(Y,dY') \leq  \dE[W]+\max_{h\leq a\leq h+\Cxi ]} \frac{\max(a,0)}{d-1}-a\leq \dE[W]+2\vert h\vert +\frac{2\Cxi }{3}\leq \frac{3\Cxi }{4}
\]
if we choose $\Cxi > 12\dE[W]+24\vert h\vert$. Combining this with~\eqref{eqn:Vdominated}, we obtain that for every $Y\in [h,h+\Cxi ]\times \cM$, 
\[
QV(Y)\leq \Cxi .
\]
Together with~\eqref{eqn:VdominatedlargeY}, this yields~\eqref{eqn:driftcondition}.
\\

\noindent
\textbf{Proof of~\eqref{eqn:doeblincondition}.} Let $v_1\ldots, v_{d-1}$ be the neighbours of $\circ$ other than $ \oc$. Let $T_1:=B_{\Td^+}(\circ,1)$ be the tree whose vertices are $\circ, v_1, \ldots, v_{d-1}$, and let $W_1:=(\circ, v_1)$.
Let  $Y\in \cX$ and let $B'$ be a Borel set of $[h,h+\Cxi ]$. Note that by~\eqref{eqn:QYsP+},
\begin{equation}\label{eqn:QlargerP+}
\begin{split}
Q(Y,B'\times \{ (T_1,W_1) \} ) &\geq  \sPr_{\varphi(Y)}(\{X_1=v_1\}\cap \{\phid(v_1)\in B'\}\cap \{\tau_1=1\}\cap \{\min_{1\leq i \leq d-1}\phid(v_i)\geq h\} )
\\
&\geq   \sP^{h,+}_{\varphi(Y)}(\{X_1=v_1\}\cap \{\phid(v_1)\in B'\}\cap \{\tau_1=1\}\cap \{\min_{1\leq i \leq d-1}\phid(v_i)\geq h\}).
\end{split}
\end{equation}
Let
\begin{equation}\label{eqn:alphadef}
\alpha:= \frac{M_hp_h^{d-2}}{d}\int_{[h,h+\Cxi ]}f_{min}(x)dx,
\end{equation}
where $M_h$ was defined at~\eqref{eqn:minprobatransience}, $ f_{min}(x):=\min_{b\in [h,h+\Cxi ]} \nu_1(x-b/(d-1))>0$ and $p_h:=\dP^{\Td}_h(\phid(v_1)\geq h)>0$. We have
\begin{equation}\label{eqn:P+largerfmin}
  \sP^{h,+}_{\varphi(Y)}(\{X_1=v_1\}\cap \{\phid(v_1)\in B'\}\cap \{\tau_1=1\}\cap \{\min_{1\leq i \leq d-1}\phid(v_i)\geq h\})\geq \alpha \int_{B'} f_{min}(x)dx.
\end{equation}
Indeed, with $ \sP^{h,+}_{\varphi(Y)}$-probability at least $p_h^{d-2}\int_{B'} f_{min}(x)dx$, $\phid(v_1)\in B'$, and a SRW on $\Chplus$ and $\phid(v_i)\geq h$ for $2\leq i\leq d-1$. Then, a SRW starting at $\circ$ has probability at least $1/d$ to jump to $v_1$, and the $\sP^{h,+}_{\varphi(Y)}$-probability that the SRW stays forever in the subtree from $v_1$ (so that $\tau_1=1$) is at least $M_h$. 
For every $B\in \cB(\cX)$, denote $B_{1,1}=\{x\in [h,h+\Cxi ],  (x,T_1,W_1)\in B \}$ and let
\begin{equation}\label{eqn:nustardef}
\nu_{\star}(B):=\frac{\int_{B_{1,1}}  f_{min}(x)dx}{\int_{B}  f_{min}(x)dx},
\end{equation}
which is clearly a probability measure on $\cX$. By~\eqref{eqn:QlargerP+},~\eqref{eqn:alphadef},~\eqref{eqn:P+largerfmin} and~\eqref{eqn:nustardef}, we have for every $B\in \cB(\cX)$:
\[
Q(Y,B)\geq   \sP^{h,+}_{\varphi(Y)}(\{X_1=v_1\}\cap \{\phid(v_1)\in B'\}\cap \{\tau_1=1\}\cap \{\min_{1\leq i \leq d-1}\phid(v_i)\geq h\}) \geq \alpha\nu_\star(B).
\]
This shows~\eqref{eqn:doeblincondition}, and the proof is complete.
\end{proof}

\begin{lemma}\label{lem:excursionsaboveK}
For $a\geq h$, let $\Gamma_a$ be the distribution of $\phid(X_{\tau_{1}})-a$ under $\sPr_a$ (which is also the distribution of $\phid(X_{\tau_{i+1}})-\phid(X_{\tau_i})$ under $\sPr_b(\,\cdot\,\vert \,\phid(X_{\tau_i})=a)$ for every $i\geq 1$ and $b\geq h$). There exists a distribution $\Gamma'$ on $\dR^+$ and $\theta'>0$ such that $\dE[e^{\theta'W}]<+\infty$ for $W\sim \Gamma'$, and such that for all $a\geq h$, 
\begin{equation}\label{eqn:excursionsaboveK}
\Gamma_a\stdom \Gamma' +\frac{a_+}{d-1}-a, \text{ with }a_+:=\max(a,0). 
\end{equation}
\end{lemma}

\begin{proof}
 By Proposition~\ref{prop:expomomentsrenewal},  we have for every $a\geq h$:
\begin{equation}\label{eqn:excusrionsaboveKheight}
\sPr_a( \vert X_{\tau_{1}}\vert  \geq m)\leq  \sPr_a( {\tau_{1}}  \geq m) \leq \Ci e^{-\Cii m}.
\end{equation}
We now bound the maximum of $\phid$ in $ B_{\Chplus}(\circ,m)$ (note that~\eqref{eqn:excusrionsaboveKheight} shows that $X_{\tau_{1}}$ is located with overwhelming probability in this ball as $m\rightarrow +\infty$). We have for every $a\geq h$ and $m\geq 1$:
\begin{align*}
&\sPr_a\left(\max_{x\in B_{\Chplus}(\circ,m)\setminus\{\circ\}}\phid(x) \geq m+\frac{a_+}{d-1}\right)
\\
&\leq M_h^{-1} \sP^{h,+,\infty}_a\left(\max_{x\in B_{\Td}(\circ,m)\setminus\{\circ\}}\phid(x) \geq m+\frac{a_+}{d-1}\right)
\\
&\leq M_h^{-1}\dP^{\Td}_h(\vert \Chplus\vert =\infty)^{-1}\dP^{\Td}_a\left(\max_{x\in B_{\Td}(\circ,m)\setminus\{\circ\}}\phid(x) \geq m+\frac{a_+}{d-1}\right). 
\end{align*}
But conditionally on $\phid(\circ)=a$, we have for all $k\geq 1$ and all $x\in \Td$ such that $\vert x\vert =k$: $\phid(x)\sim \frac{a}{(d-1)^k}+Y$, where $Y\sim \cN\left(0,\frac{d-1}{d-2}(1-(d-1)^{-2k})\right)$. Thus, by the exponential Markov inequality for a centred Gaussian variable, noticing that $\frac{a}{(d-1)^k}\leq \frac{a_+}{d-1} $ and that $\text{Var}(Y)\leq 2$ for all $k\geq 1$, we have
\[
\dP^{\Td}_a\left(\phid(x)\geq m+\frac{a_+}{d-1}\right)\leq \dP(Y\geq m)\leq \exp(-m^2/(2\text{Var}(Y)))\leq \exp(-m^2/8).
\]
Since $\vert B_{\Td}(\circ,m)\vert\leq d^m$, by a union bound for $x\in B_{\Td}(\circ,m)$, we thus obtain that if $C>0$ is large enough (depending only on $d$ and $h$), then 
\[
\sPr_a\left(\max_{x\in B_{\Chplus}(\circ,m)\setminus\{\circ \}}\phid(x) \geq m+\frac{a_+}{d-1}\right)\leq ce^{-\Cii m}.
\]

\noindent
Combining this with \eqref{eqn:excusrionsaboveKheight} yields
\[
\sPr_a\left(\phid(X_{\tau_{1}})-a  \geq m+\frac{a_+}{d-1}-a \right)\leq (c+\Ci)e^{-\Cii m}.
\]
The conclusion follows.
\end{proof}

%
%

\subsection{Pointwise LLN and CLT}\label{subsec:pointwise}
In this Section, we establish the following result, which will provide the convergence for finite-dimensional marginals of the processes in Theorem~\ref{thm:LLN}. 

\begin{proposition}\label{prop:LLNandCLT}
For every $h<h_{\star}$, there exists constants $s_h,\sigma_h>0$ such that if $(X_k)_{k\geq 0}$ is a SRW on $\Ch$ started at $\circ$, then
\begin{equation}\label{eqn:propLLN}
\frac{\vert X_k\vert }{k}\,\overset{\sPr_h-a.s.}{\longrightarrow}\,s_h,
\end{equation}
and under~$\sPr_h$,
\begin{equation}\label{eqn:propCLT}
\frac{\vert X_k\vert -s_hk}{\sqrt{k}}\overset{(d)}{\longrightarrow} \cN(0,\sigma_h^2). 
\end{equation}
\end{proposition}

\begin{proof}[Proof of~\eqref{eqn:propLLN}.] We establish~\eqref{eqn:propLLN} with $s_h=s_{h,X}/s_{h,\tau}$. For $k\geq 1$, define $\theta_k:=\max\{i\geq 0, \tau_i \leq k\}$. Note that the LLN~\eqref{eqn:CLThandtau} implies that
\begin{equation}\label{eqn:LLNthetainversetau}
\theta_k/k \overset{\sPr_h-a.s.}{\longrightarrow} 1/s_{h,\tau}
\end{equation}
We have 
\begin{align*}
\left\vert\frac{ \vert X_k\vert }{k}- \frac{s_{h,X}}{s_{h,\tau}}\right\vert &\leq \left\vert\frac{ \vert X_k\vert }{k}-\frac{\vert X_{\tau_{\theta_k}}\vert }{k}\right\vert + \left\vert \frac{\vert X_{\tau_{\theta_k}}\vert }{k}- \frac{s_{h,X}}{s_{h,\tau}}\right\vert 
\\
&\leq \left\vert 1-\frac{ \tau_{\theta_k}}{k}\right\vert +  \left\vert \frac{\vert X_{\tau_{\theta_n}}\vert }{n}- \frac{s_{h,X}}{s_{h,\tau}}\right\vert 
\\
&\leq\frac{ \max_{1\leq k\leq n+1}(\tau_k-\tau_{k-1})}{n} +  \left\vert \frac{\vert X_{\tau_{\theta_n}}\vert }{n}- \frac{s_{h,X}}{s_{h,\tau}}\right\vert.
\end{align*}
By Proposition~\ref{prop:expomomentsrenewal} and a union bound over $k$ we have for $k$ large enough: 
\begin{center}
$\sPr_h(\max_{1\leq i\leq k+1}(\tau_i-\tau_{i-1}) \geq k^{1/3})\leq (k+1)\Ci e^{-\Cii k^{1/18}}\leq k^{-5}.$
\end{center}
Applying Borel-Cantelli's Lemma, we have that $\sPr_h$-a.s., 
\begin{equation}\label{eqn:maxtaubounded}
\left\vert\frac{ \vert X_k\vert }{k}-\frac{\vert X_{\tau_{\theta_k}}\vert }{k}\right\vert\leq    \left\vert 1-\frac{ \tau_{\theta_k}}{k}\right\vert \leq \frac{ \max_{1\leq i\leq k+1}(\tau_i-\tau_{i-1})}{k}\leq k^{-2/3}
\end{equation}
for $k$ large enough. Thus, we only have to show that $\sPr_h$-a.s.
\begin{equation}\label{eqn:LLNXtauthetan}
 \left\vert \frac{\vert X_{\tau_{\theta_k}}\vert }{k}- \frac{s_{h,X}}{s_{h,\tau}}\right\vert \rightarrow 0.
\end{equation}
Let $\varepsilon \in (0,1/10)$. Then~\eqref{eqn:CLThandtau},~\eqref{eqn:CLThandtauX} and~\eqref{eqn:LLNthetainversetau} imply that $\sPr_h$-a.s., there exists $k_0\in \dN$ such that for all $k\geq k_0$,
\begin{equation}
\begin{split}
(1-\varepsilon)s_{h,\tau}k \leq& \tau_k \leq(1+\varepsilon)s_{h,\tau}k,\,\,\,(1-\varepsilon)s_{h,X}k \leq \vert X_{\tau_k}\vert  \leq(1+\varepsilon)s_{h,X}k 
\\
\text{ and }& (1-\varepsilon)s_{h,\tau}^{-1}k\leq \theta_k \leq (1+\varepsilon)s_{h,\tau}^{-1}k,
\end{split}
\end{equation}
so that $\vert X_{\tau_{\theta_k}}\vert \leq (1+\varepsilon)^3\frac{s_{h,X}}{s_{h,\tau}}k\leq (1+7\varepsilon)\frac{s_{h,X}}{s_{h,\tau}}k$. Similarly, we obtain $\vert X_{\tau_{\theta_k}}\vert \geq  (1-7\varepsilon)\frac{s_{h,X}}{s_{h,\tau}}k$, so that for every $k\geq 2s_{h,\tau}k_0$, 
\begin{equation}\label{eqn:LLNXtauthetanproved}
 \left\vert \frac{\vert X_{\tau_{\theta_k}}\vert }{k}- \frac{s_{h,X}}{s_{h,\tau}}\right\vert \leq 6\varepsilon  \frac{s_{h,X}}{s_{h,\tau}}.
\end{equation}
Since $\varepsilon >0$ was arbitrary, the conclusion follows. 
\end{proof}

\noindent
\begin{proof}[Proof of~\eqref{eqn:propCLT}.] By~\eqref{eqn:maxtaubounded}, it is enough to show~\eqref{eqn:propCLT} for $X_{\tau_{\theta_k}}$ instead of $X_k$. We have 
\begin{equation}\label{eqn:CLTfirstfrac}
\frac{\vert X_{\tau_{\theta_k}}\vert }{k}-s_h =\frac{\vert X_{\tau_{\theta_k}}\vert }{\tau_{\theta_k}}\times \frac{\tau_{\theta_k}}{k} -s_h. 
\end{equation}
Defining
\begin{equation}\label{eqn:Witildeidef}
\tilde{\tau}_i:=\tau_i-\tau_{i-1}-s_{h,\tau},\,  \widetilde{X}_i:=\vert X_{\tau_{i}}\vert -\vert X_{\tau_{i-1}}\vert -s_{h,X}\text{ and }W_i=\frac{\widetilde{X}_i}{s_{h,X}}-\frac{\tilde{\tau}_i}{s_{h,\tau}}
\end{equation}
for $i\geq 1$, we remark that
\begin{equation}\label{eqn:fracdevelop}
\frac{ \vert X_{\tau_{\theta_k}}\vert }{\tau_{\theta_k}}=\frac{\theta_ks_{h,X}+\sum_{i=1}^{\theta_k} \widetilde{X}_i }{\theta_ks_{h,\tau}+\sum_{i=1}^{\theta_k}\tilde{\tau}_i} =s_h\frac{1+\theta_k^{-1} \sum_{i=1}^{\theta_k} \widetilde{X}_i/s_{h,X} }{1+\theta_k^{-1} \sum_{i=1}^{\theta_k} \tilde{\tau}_i/s_{h,\tau}}= s_h+\frac{s_h}{\theta_k} \sum_{i=1}^{\theta_k}  \left(\frac{\widetilde{X}_i}{s_{h,X}}-\frac{\tilde{\tau}_i}{s_{h,\tau}} \right) +r_k
\end{equation}
with $r_k=o(k^{-2/3})$ $\sPr_h$-w.h.p. Indeed, the CLTs~\eqref{eqn:CLThandtau} and~\eqref{eqn:CLThandtauX} ensure that $\sPr_h$-w.h.p., $\theta_k^{-1} \sum_{i=1}^{\theta_k} \tilde{\tau}_i=o(\theta_k^{-1/3})$ and $\theta_k^{-1} \sum_{i=1}^{\theta_k} \widetilde{X}_i=o(\theta_k^{-1/3})$, and we have $\liminf_{k\rightarrow +\infty}\theta_k/k >0$ by the LLN~\eqref{eqn:CLThandtau}. This ensures that we can stop the development of the fraction in~\eqref{eqn:fracdevelop} at the first order. Letting $W_i:=\frac{\widetilde{X}_i}{s_{h,X}}-\frac{\tilde{\tau}_i}{s_{h,\tau}}$ for $i\geq 1$, we thus have
\[
\frac{ \vert X_{\tau_{\theta_k}}\vert }{\tau_{\theta_k}}=s_h+\frac{1}{\theta_k}\sum_{i=1}^{\theta_k}W_i +r'_k.
\]
Note that $\theta_k^{-1/2}\sum_{i=1}^{\theta_k}W_i =O(1)$ $\sPr_h$-w.h.p.~by the CLTs~\eqref{eqn:CLThandtau} and~\eqref{eqn:CLThandtauX}, and that $\tau_{\theta_k}/k= 1+r'_k$ with $r'_k=O(k^{-2/3})$ $\sPr_h$-a.s.~by~\eqref{eqn:maxtaubounded}. Therefore,~\eqref{eqn:CLTfirstfrac} becomes
\begin{equation}\label{eqn:CLTbidouille}
\sqrt{k}\left( \frac{X_{\tau_{\theta_k}}}{k}-s_h  \right)=\sqrt{\frac{k}{\theta_k}} \frac{1}{\sqrt{\theta_k}}\sum_{i=1}^{\theta_k}W_i \,\, +r''_k=  \frac{\sqrt{\tau}}{\sqrt{\theta_k}}\sum_{i=1}^{\theta_k}W_i +r^{(3)}_k
\end{equation}
with $r''_k=o(1)$ and $r^{(3)}_k=o(1)$  $\sP^{h,+,\infty}_a$-w.h.p., and where we have used~\eqref{eqn:LLNthetainversetau} for the second equality. The proof of Proposition~\ref{prop:CLThtau} applies straightforwardly when replacing the sequences $(\tau_{i+1}-\tau_i)_{i\geq 1}$ and $(\vert X_{\tau_{i+1}}\vert -\vert X_{\tau_i}\vert )_{i\geq 1}$ by any of their affine combinations, in particular $W_i$. By~\eqref{eqn:Witildeidef} and by definition of $s_{h,X}$ and $s_{h,\tau}$ in Proposition~\ref{prop:CLThtau}, this entails the existence of $\sigma_{h,W}\geq 0$ such that under $\sPr_h$, 
\[
\frac{1}{\sqrt{k}}\sum_{i=1}^kW_i \overset{(d)}{\longrightarrow} \cN(0,\sigma_{h,W}^2)
\]
when $k\rightarrow +\infty$. As $\sPr_h$-a.s., $\theta_k\rightarrow +\infty$ as $k\rightarrow +\infty$, this combined to~\eqref{eqn:CLTbidouille} yields the CLT with $\sigma_h=\sqrt{\tau}\sigma_{h,W}$, and it only remains to check that $\sigma_{h,W}>0$. 
\\
\\
\textbf{Positivity of $\sigma_{h,W} $.} It is enough to prove that the variance of $\sum_{i=1}^kW_i$ grows at least linearly in $k$. In a nutshell, the variance of $W_i$ is bounded away of $0$ as soon as the renewal interval $Y_i$ has height at least $2$ (becase the SRW can 'wiggle' on its way from $X_{\tau_{i}}$ to $X_{\tau_{i+1}}$). This has a positive $\pi$-probability, hence a positive proportion of the first $k$ renewal intervals will satisfy this property w.h.p.~as $k\rightarrow+\infty$. To eliminate the effect of covariances, we use that the $W_i$'s are independent conditionally on the values of the $\phid(X_{\tau_i})$'s. 
\\
Let $\cM_2:=\{(T,W)\in \cM,\, \h(T)=2\text{ and $T$ has exactly two edges}\}$. For $k\geq 1$, let $\overline{W}_k$ be the distribution of $(\phid(X_{\tau_i}), \mathbf{1}_{\{Y_i\in \cM_2\}})_{0\leq i\leq k}$ under $\sPr_h$, and denote $\overline{\dE}_k$ the corresponding expectation. By the total variance formula and Proposition~\ref{prop:expomomentsrenewal} (which gives the independence of renewal intervals conditionally on the values of $\phid$ on the endpoints of each interval), we have 
\begin{equation}\label{eqn:variancetotalformula}
\text{Var}_{\sPr_h}\left(\sum_{i=1}^kW_i\right)\geq \overline{\dE}_k\left[\text{Var}_{\sPr_h}\left(\sum_{i=1}^k W_i\,\bigg|\, \overline{W_k}\right)\right] =\sum_{i=1}^k\text{Var}_{\sPr_h}\left(W_i\,\vert\, \overline{W_k}\right).
\end{equation}
Let $K>0$ be large enough such that $\pi([h,h+K]\times \cM)>0$. By~\eqref{eqn:reachabilitypositive} applied to $\cM_2$ (which is non-empty) instead of a fixed $(T,W)\in \cM$ and $I=[h,+\infty)$, there exists $\varepsilon>0$ such that $\pi([h,+\infty))> \varepsilon$. 
By Proposition~\ref{prop:CLThtau} applied to $f=\mathbf{1}_{Y\in [h,+\infty)\times \cM_2}$, there exists $k_0$ large enough such that for all $k\geq k_0$, 
\begin{equation}\label{eqn:varianceconditional}
\overline{\dP}_k( \vert \{i\leq k, \, Y_i\in [h,+\infty)\times \cM_2  \} \vert \geq \varepsilon k)\geq 1/2.
\end{equation}
Note that conditionally on $Y_i\in \cM_2$, $W_i\sim \frac{2}{s_{h,X}}- \frac{2U}{s_{h,\tau}}$ where $U\sim \text{Geom}(1/2)$, since in this case, $\widetilde{X}_i=h(Y_i)=2$ and the SRW does a geometric number of back-and-forths on the first edge of the renewal interval, before crossing the second edge only once (recall that by definition of renewal intervals, recall that the SRW goes once through $(X_{\tau_j-1},X_{\tau_j})$ for every $j\geq 1$). Hence the conditional variance of $W_i$ is bounded by below by some constant $\delta >0$. Therefore, for $k\geq k_0$,~\eqref{eqn:variancetotalformula} and~\eqref{eqn:varianceconditional} give
\[
\text{Var}_{\sPr_h}\left(\sum_{i=1}^kW_i\right)\geq \frac{1}{2}\times (\varepsilon k)\times \delta \geq \frac{\delta \varepsilon}{2}k,
\]
so that $\sigma_{h,W}\geq  \delta\varepsilon/2>0$. This concludes the proof. 
\end{proof}

\subsection{Proof of Theorem~\ref{thm:LLN}}\label{subsec:proofthmmain}

\begin{proof}[Proof of~\eqref{eqn:thmLLN}] The proof simply combines~\eqref{eqn:propLLN} with the monotonicity of the identity function and a classical diagonal argument. Let $\varepsilon\in (0,1)$. By Proposition~\ref{prop:expomomentsrenewal} and a union bound, when $k$ is large enough, then
\begin{center}
$\sPr_h(\cE_1(k))\geq 1-k^{-100}$  where $\cE_1(k):=\{\sup_{1\leq i\leq k}\tau_{i+1}-\tau_i \leq \varepsilon k/2 \}$. 
\end{center}
By Borel-Cantelli's Lemma,  there exists $\sPr_h$-a.s.~a (random) $k_0\in \dN$ such that $\cap_{k\geq k_0}\cE_1(k)$ holds. 
\\
Note also that~\eqref{eqn:propLLN} implies the $\sPr_h$-a.s.~existence of a (random) $k_0'\in \dN$ such that for all $k\geq k_0'$,  $\cE_2(k):=\{\max_{1\leq j \leq \lceil \varepsilon\rceil +3 }\vert \vert X_{\lfloor j\varepsilon k\rfloor}\vert /k -s_h j\varepsilon\vert \leq s_h\varepsilon \}$ holds.  
On $\cE_1(k)\cap \cE_2(k)$, for every $t\in [0,1]$, 
there exists $i\geq 0$ and $j\in [0,\lceil \varepsilon\rceil+3]$ such that $0\leq j\varepsilon k\leq \tau_i\leq  \lfloor kt\rfloor\leq (j+2)\varepsilon k$ and $j\varepsilon k\leq \lfloor kt\rfloor\leq (j+1)\varepsilon k$. Hence $\vert X_{\lfloor kt\rfloor}\vert \geq \vert X_{\tau_i}\vert \geq \vert X_{\lfloor j\varepsilon k\rfloor}\vert \geq s_h j\varepsilon k-s_h\varepsilon k\geq s_hkt-3s_h\varepsilon k$. We obtain similarly $\vert X_{\lfloor kt\rfloor}\vert \leq s_hkt+3s_h\varepsilon k$. 
\\
Thus, we have shown that for every $\varepsilon \in (0,1)$, there exists $\sPr_h$-a.s.~$k_0'':= k_0+k_0'$ so that we have for every $k\geq k_0''$: 
\begin{center}
$\sup_{0\leq t \leq 1} \vert\, \vert X_{\lfloor kt\rfloor}\vert /k- s_h t \vert \leq 4s_h\varepsilon $. 
\end{center}
Applying this argument to the sequence $(\varepsilon^m)_{m\geq 1}$ (which converges to $0$) and using that a countable intersection of sets of full measure is still of full measure yields the result.
\end{proof}

\begin{proof}[Proof of~\eqref{eqn:thmCLT}]
We proceed in two steps. First, we show the convergence of the finite-dimensional marginals, then we establish the tightness of the sequence $\left((\vert X_{\lfloor kt\rfloor}\vert )_{0\leq t\leq 1}\right)_{k\geq 0}$.
\\
\textbf{Finite dimensional marginals.} We claim that for every integer $m\geq 2$ and all $0<t_1<t_2<\ldots < t_m\leq 1$, under $\sPr_h$, 
\begin{equation}\label{eqn:fddCLT}
k^{-1/2}(\vert X_{\lfloor kt_1\rfloor }\vert -s_hkt_1, \ldots,\vert  X_{\lfloor kt_m\rfloor }\vert -s_hkt_m ) \overset{(d)}{\longrightarrow} (B_{t_1}, \ldots, B_{t_m}),
\end{equation}
where $B$ is a standard Brownian motion. 
For the sake of simplicity, we restrict ourselves to the case $m=2$ (the generic case will follow straightforwardly from our proof). We can reformulate~\eqref{eqn:fddCLT} as
\begin{equation}\label{eqn:fddCLT2coord}
k^{-1/2}(\vert X_{\lfloor kt_1\rfloor }\vert -s_hkt_1, \vert X_{\lfloor kt_2\rfloor }\vert -\vert X_{\lfloor kt_1\rfloor }\vert -s_hk(t_2-t_1))   \overset{(d)}{\longrightarrow}(W_1,W_2),
\end{equation}
where $(W_1,W_2)$ is a pair of independent centred Gaussian variables of variance $t_1$ and $t_2-t_1$ respectively.
By~\eqref{eqn:propCLT}, we already know that
\begin{equation}\label{eqn:fddCLT1stcoord}
k^{-1/2}(\vert X_{\lfloor kt_1\rfloor }\vert -s_hkt_1) \overset{(d)}{\longrightarrow}W_{1}.
\end{equation}
To show that $\vert X_{\lfloor kt_2\rfloor }\vert -\vert X_{\lfloor kt_1\rfloor }\vert $ is asymptotically independent of $\vert X_{\lfloor kt_1\rfloor }\vert $, we prove that with high probability, the SRW $(X_j)$ has a renewal time $\tau \in [\lfloor kt_1\rfloor+1, \lfloor kt_1\rfloor +k^{1/3}]$ with GFF value in $[h,h+\Cxi ]$ and the next renewal interval is taken according to $\nu_\star$ (which happens after a geometric number of visits to $[h,h+\Cxi ]$ at renewal times, by~\eqref{eqn:doeblincondition}). This allows the SRW to forget about its GFF value at time $\lfloor kt_1\rfloor$. We then apply~\eqref{eqn:propCLT} to $\vert X_{\lfloor kt_2\rfloor}\vert -\vert X_{\tau}\vert  $ under $\sPr_{\nu_\star^{(1)}}$, where $\nu_\star^{(1)}$ is the projection of $\nu_\star$ onto its first coordinate. Since $\vert \,\,\vert X_{\tau}\vert -\vert X_{\lfloor kt_1\rfloor}\vert \,\,\vert \leq \tau-\lfloor kt_1\rfloor=o(k^{1/2})$, this will conclude the proof of~\eqref{eqn:fddCLT2coord}.
\\
\\
In detail, let $i_0$ be the smallest positive integer such that $\tau_{i_0}\geq kt_1$ (we drop deliberately the dependency in $k$ in the notation). By~\eqref{eqn:doeblincondition}, we can realize $(X_j)_{j\geq \tau_{i_0}}$ in the following way, conditionnally on the value of $\phid(X_{\tau_{i_0}})$. Recusively for $i\geq i_0$, if $\phid(X_{\tau_{i}})>h+\Cxi $, we pick the renewal interval between $\tau_i$ and $\tau_{i+1}$ according to the distribution of the first renewal interval under $\sPr_{\phid(X_{\tau_i})}$. If $\phid(X_{\tau_{i}})\in [h,h+\Cxi ]$, let $U_i$ be a uniform random variable in $[0,1]$, independent from everything else. If $U_i\leq \alpha$, we sample $\phid(X_{\tau_{i+1}})$ according to $\nu_\star^{(1)}$. If $U_i>\alpha$, we sample $\phid(X_{\tau_{i+1}})$ according to an ad hoc probability measure $\nu_\star^{(1),\phid(X_{\tau_i})}$, which depends on $\phid(X_{\tau_i})$ (but not on $U_i$) and whose existence is guaranteed by~\eqref{eqn:doeblincondition}. Then, we sample the renewal interval between $\tau_i$ and $\tau_{i+1}$ according to the ad hoc distribution of a renewal interval conditionally on the value of $\phid$ at its extremities.
\\
Let $i_1:=\inf\{ i\geq 1, \, \phid(X_{\tau_i})\in [h,h+\Cxi ],\, U_i<\alpha\}$ . Then $(X_j)_{j\geq \tau_{i_1}}$ is distributed as the SRW on $\Chplus$ under $\sPr_{\nu_\star^{(1)}}$. As mentioned below~\eqref{eqn:fddCLT1stcoord}, we can apply~\eqref{eqn:propCLT} to $\vert X_{\lfloor kt_2\rfloor}\vert -\vert X_{\tau_{i_1}}\vert  $, so that to establish~\eqref{eqn:fddCLT2coord}, it suffices to prove that for $k$ large enough,
\begin{equation}\label{eqn:fddrenewalfast}
\sPr_h (\tau_{i_1}\leq kt_1+ k^{1/4}) \geq 1 - 1/k. 
\end{equation}
Let $\cE_1:=\{\max_{0\leq i\leq k} \tau_{i+1}-\tau_i \leq k^{1/10} \}\cap \{\max_{x\in B_{\Td}(\circ, k)}\phid(x)< k^{1/100} \}$. By Proposition~\ref{prop:expomomentsrenewal} for the first event, and by the exponential Markov inequality applied to a centred Gaussian variable of variance $(d-1)/(d-2)$, we have for $k$ large enough:  
\begin{equation}\label{eqn:fddprobaE1}
\sPr_h(\cE_1)\geq 1- k \Ci e^{-\Cii k^{1/60}} - d^k \exp\left( - \frac{k^{2/100}}{2(d-1)/(d-2)}\right)\geq 1- k^{-100} . 
\end{equation}
For $j\geq 1$, let $i'_j$ be the $j$-th renewal time after $kt_1$ such that $\phid(X_{\tau_{i'_j}})\leq h+\Cxi $. Note that we can choose $\Cxi >\frac{d-1}{d-2}(1+\dE[\Gamma']+\vert h\vert)$ in its definition in Lemma~\ref{lem:Vpotential}. Then, we apply~\eqref{eqn:excursionsaboveKbis} with $\varphi(Y)\leq k^{1/100}$, $K=\Cxi $ and $m\geq k^{1/99} $, and we have for all $k$ large enough:
\begin{equation}
\sPr_h(\cE_1 \cap \{ \max_{1\leq j\leq k }\tau_{i'_{j+1}}-\tau_{i'_j} \geq  k^{1/9}\}) \leq \dP^{\Td,\infty}(\cE_1 \cap \{\max_{1\leq j\leq k }i'_{j+1}-i'_j \geq  k^{1/99}\})\leq  k^{-10}.
\end{equation}
Similarly, noticing that $\tau_{i_0}\leq kt_1+k^{1/10}$ and $\phid(X_{\tau_{i_0}})<k^{1/100}$ on $\cE_1$, we have 
\begin{equation}
\sPr_h(\cE_1 \cap\{\tau_{i'_{1}} \geq kt_1 + k^{1/9}\} ) \leq \dP^{\Td,\infty}(\cE_1 \cap \{ \tau_{i'_{1}}\geq \tau_{i_0}-k^{1/10}+ k^{1/9}  \})\leq  k^{-10}.
\end{equation}
Combining these two estimates with~\eqref{eqn:fddprobaE1}, this yields 
\begin{equation}\label{eqn:fddclosecoord}
\sPr_h(\tau_{i'_{\lfloor \log^2 k\rfloor}}\leq kt_1 +k^{1/4} )\geq 1- k^{-2}. 
\end{equation}
Finally, we have 
\begin{equation}
\sPr_h(\tau_{i_1}\geq \tau_{i'_{\lfloor \log^2 k\rfloor}} )\leq (1-\alpha)^{\lfloor\log^{2}k\rfloor}\leq k^{-2}
\end{equation}
for $k$ large enough. Together with~\eqref{eqn:fddclosecoord}, this yields~\eqref{eqn:fddrenewalfast} and thus~\eqref{eqn:fddCLT2coord}.

\noindent
\textbf{Tightness.} We apply Theorems 17.4.2 and 17.4.4 from~\cite{MeynTweedie} to show that the rescaled sequence $(k^{-1/2}(\vert X_{\lfloor \tau_kt\rfloor}\vert-s_{h,X}kt)_{0\leq t\leq 1} )_{k\geq 1}$ satisfies a Donsker theorem. Thus, it enjoys regularity properties that we translate to  $(k^{-1/2}(\vert X_{\lfloor kt\rfloor}\vert-s_hkt )_{0\leq t\leq 1} )_{k\geq 1}$, using the stretched exponential bound on renewal times from Proposition~\ref{prop:expomomentsrenewal}.
\\
Precisely, it is enough to show that for any $\varepsilon >0$, there exists $\delta >0$ small enough such that for $k$ large enough, 
\begin{equation}\label{eqn:tightnessquantify}
\sPr_h (\max_{1\leq j\leq k, 1\leq i\leq \delta\sqrt{k}}\vert\, \vert X_{j+i}\vert-  \vert X_j\vert  -s_h i  \, \vert \leq \varepsilon\sqrt{k} ) \geq 1-\varepsilon. 
\end{equation}
Fix $\varepsilon >0$. By Theorem 17.4.2 of~\cite{MeynTweedie} and~\eqref{eqn:CLThandtauX}, the assumptions of Theorem 17.4.4 of~\cite{MeynTweedie} hold for the Markov chain $(Y_i)_{i\geq 1}$, and the maps $g(Y):=h(Y)-s_{h,X} $ and $g'(Y):=\tau(Y)-s_{h,\tau}$. As a consequence of this theorem, both sequences
\begin{center}
$k^{-1/2}\sigma_{h,X}^{-1}(\vert X_{\lfloor \tau_kt\rfloor}\vert-s_{h,X}kt)_{0\leq t\leq 1}$ and  $k^{-1/2}\sigma_{h,\tau}^{-1}(\vert \tau_{\lfloor kt\rfloor}\vert-s_{h,\tau}kt)_{0\leq t\leq 1}$ 
\end{center}
converge in distribution to standard Brownian motions on $[0,1]$. Using that almost every realization of the Brownian motion is uniformly continuous and recalling~\eqref{eqn:fddprobaE1},  
one gets easily the existence $\delta \in (0,\varepsilon)$ small enough such that for every $k$ large enough,
\begin{equation}\label{eqn:unifcontCLTtension}
\begin{split}
&\sPr_h(\cE_1\cap \cE_2)\geq 1-\varepsilon/2\text{, where }
\\
&\cE_2:=\left\{   \max_{1\leq j\leq k, 1\leq i\leq 2(s_{h,\tau}^{-1}+1)\delta\sqrt{k}}\vert\,  \vert X_{\tau_{j+i}}\vert -\vert X_{\tau_j}\vert  -s_{h,X}i  \, \vert +   \vert\,  {\tau_{j+i}}-{\tau_j} -s_{h,\tau}i  \, \vert        \leq \varepsilon\sqrt{k}/3  \right\}. 
\end{split}
\end{equation}
Assume now that $\cE_1\cap \cE_2$ holds for some fixed values of $\delta$ and $k$. Take $j\in \{1, \ldots, k \}$ and $i\in \{1, \ldots , \lfloor \delta \sqrt{k}\rfloor\}$. Let $j'\geq 1$ and $i'\geq 0$ be such that $\tau_{j'}$ (resp. $\tau_{j'+i'}$) is the smallest renewal time larger or equal to $j$ (resp. $j+i$). We have
\begin{align*}
\vert\, \vert X_{j+i}\vert -\vert X_j\vert - s_h i\,\vert  \leq & \vert j-\tau_{j'}\vert +\vert (j+i)-\tau_{j'+i'}\vert + \vert\, \vert X_{\tau_{j'+i'}}\vert -\vert X_{\tau_{j'}}\vert - s_h i\,\vert 
\\
\leq & k^{1/3}+  \vert\, \vert X_{\tau_{j'+i'}}\vert -\vert X_{\tau_{j'}}\vert - s_{h,X}i'\,\vert + \vert  s_{h,X}i' -  s_h i \,\vert 
\\
\leq & k^{1/3}+\varepsilon\sqrt{k}/3 +  s_{h,X} \vert \, i/s_{h,\tau} - i'\,\vert
\end{align*}
where the last inequality comes from~\eqref{eqn:unifcontCLTtension}. Indeed, we have $1\leq j'\leq k$ and $1\leq i'\leq i+1  \leq 2\delta \sqrt{k}$ (there are at most $ (j+i -j')+1\leq i+1$ renewal times between $j'$ and $j'+i'$ since $j'\geq j$ and by definition of $j'+i'$). Hence it only remains to show that 
\begin{equation}\label{eqn:iprimeDonsker}
\vert \, i/s_{h,\tau} - i'\,\vert \leq s_{h,\tau}^{-1}\varepsilon \sqrt{k}/2,
\end{equation}
which will follow from 
\begin{equation}\label{eqn:iprimeDonskertau}
\tau_{j'+i_1 } \leq  \tau_{j'+i' } \leq\tau_{j'+i_2}\text{ with }i_1=s_{h,\tau}^{-1}i- s_{h,\tau}^{-1}\varepsilon\sqrt{k}/2\text{ and }i_2= s_{h,\tau}^{-1}i+s_{h,\tau}^{-1}\varepsilon\sqrt{k}/2.
\end{equation}
We only prove the right inequality (as the left one can be showed in a similar way). We have $j'\leq k$ and $i_2\leq s_{h,\tau}^{-1}(\delta+\varepsilon)\sqrt{k} \leq 2(s_{h,\tau}^{-1}+1)\sqrt{k}$ since we chose $\delta <\varepsilon$. Thus we can apply~\eqref{eqn:unifcontCLTtension} and obtain that 
\[
\tau_{j'+i_2} - \tau_j'\geq s_{h,\tau}i_2 - \varepsilon\sqrt{k}/3\geq i+\varepsilon\sqrt{k}/6.
\]
This implies that $\tau_{j'+i_2}\geq \tau_j'+i+\varepsilon\sqrt{k}/6\geq j+i +\varepsilon\sqrt{k}/6\geq \tau_{j'+i'}$ by  and by definition of $j'+i'$, so that~\eqref{eqn:iprimeDonskertau} follows. This concludes the proof of~\eqref{eqn:iprimeDonsker} and~\eqref{eqn:tightnessquantify}, and thus of the theorem.
\end{proof}

\begin{appendix}
\section{Appendix}\label{appn}

\subsection{Proof of Proposition~\ref{lem:grimmettkesten}}\label{subsec:appendixgrimmettkesten}

Proposition~\ref{lem:grimmettkesten} is analogous to Lemma 1 in~\cite{GrimmettKesten} for Galton-Watson trees. In~\cite{GrimmettKesten}, the argument relies on the fact that $F'(q)<1$, where $F$ is the generating function associated to the reproduction law of a supercritical Galton-Watson tree (i.e. every individual has in average $>1$ children), and $q$ is the extinction probability. In fact, $F'(q)$ corresponds to the following quantity: for a given vertex $z$, it is the sum over each child $z'$ of $z$ of the probability that for every other child $z''$ of $z$, the subtree $T(z'')$ is finite. We want to transpose this to our setting of an infinite-type branching process. 
\\
The analogue of $F$ is the operator $R_h$ defined in \eqref{eqn:Rhdef}, and the analogue of $q$ is the function $q_h$ defined in \eqref{eqn:qhdef}. By Proposition~3.6 of \cite{ACregultrees}, at any $f\in L^2(\nu)$, $R_h$ has a Fréchet derivative $A_h^f$ given by 
\begin{equation}\label{eqn:Ahfg}
A_h^fg(a)= \mathbf{1}_{[h,+\infty)}(a)\cdot(d-1)\dE_Y\left[f\left(\frac{a}{d-1}+Y\right)\right]^{d-2}\dE_Y\left[g\left(\frac{a}{d-1}+Y\right)\right]^{}
\end{equation}
for $a\in \dR$, $g\in L^2(\nu)$ and $Y\sim \nu_1$. Then for $a\in\dR$ and $\delta\in (0,1)$, the analogue of $F'(q)$ in our context is $A_h^{q_{h,\delta}}g_h(a) $,  where we recall that $g_h:=\mathbf{1}_{[h,+\infty)}.$ 
\\
The quantity $A_h^{q_{h,\delta}}g_h(a) $ is the sum over the children $z$ of $\circ$ in $\Td^+$ of the $\dP^{\Td}_a$-probability that all the subtrees rooted at other children of $\circ$ in $\Td^+$ are not $\delta$-transient. Unfortunately, it absolutely not clear that there exists $\delta\in (0,1)$ such that $\sup_{a\geq h} A_h^{q_{h,\delta}}g_h(a)<1$. To remedy this, we proceed to a finite scaling, by looking at the $k$-offspring of $\circ$ for some $k$ large enough, instead of the children of $\circ$ (Lemma~\ref{lem:Akhsmall} below). 
\\

\noindent
For every $k\geq 2$, by the chain rule and a straightforward induction, $R_h^k$ has a Fréchet derivative $A_{k,h}^f:=A_h^{R_h^{k-1}f}\circ A_{k-1,h}^f$ at any $f\in L^2(\nu)$, so that $A_{k,h}$ is $A_h$ iterated $k$ times. 


\begin{lemma}\label{lem:Akhsmall}
Fix $\delta\in (0,\delta_0)$, where $\delta_0$ was defined in Proposition~\ref{prop:transienceintro}. There exists $\epsilon >0$ such that for large enough $k$, 
\begin{equation}\label{eqn:fprimeq}
\sup_{a\geq h}\vert A_{k,h}^{q_{h,\delta}}g_h(a) \vert <(1-\epsilon)^k.
\end{equation}
\end{lemma}

\begin{proof}
For $\varepsilon>0$ and $k\geq 1$, one has by (\ref{eqn:Rhnintermsofchildren}) and the fact that $g_h(u)=1$ for all $u\geq h$:
\begin{align*}
R_h^k(q_{h,\delta}+&\varepsilon g_h)(a)=  \dE^{\Td}_a\left[\prod_{y\in \cZ_k^{h,+}}(q_{h,\delta}+\varepsilon g_h)(\phid(y)) \right]
\\
&=  R_h^kq_{h,\delta}(a)+  \varepsilon\dE^{\Td}_a\left[\sum_{y\in \cZ_k^{h,+}}\prod_{y'\in \cZ_k^{h,+}\setminus \{y\}}q_{h,\delta}(y') \right]+o_{\varepsilon\rightarrow 0}(\varepsilon).
\end{align*}
By (\ref{eqn:expomomentssize}) applied to $\cZ_k^{h,+}$, there exists a constant $C>0$ such that if $k$ is large enough, we have $\sup_{a\geq h}\dP_a^{\Td}(1\leq \vert \cZ_{k}^{h,+}\vert \leq k^2)\leq e^{-Ck}$. Moreover, $q_{h,\delta_1}$ is clearly non-negative, and is non-negative. Thus, $\sup_{b\geq h}\vert q_{h,\delta}(b)\vert= q_{h,\delta}(h)<1$. Hence, for $k$ large enough and for every $a\geq h$, 
\[ 
0\leq \dE^{\Td}_a\left[\sum_{y\in \cZ_k^{h,+}}\prod_{y'\in \cZ_k^{h,+}\setminus \{y\}}q_{h,\delta}(y') \right]\leq k^2e^{-Ck}+\max_{j\geq k^2}jq_{h,\delta}(h)^{j-1} \leq e^{-Ck/2}.
\]
Therefore,
\[
\vert R_h^k(q_{h,\delta}+\varepsilon g_h)(a)- R_h^k(q_{h,\delta})(a)\vert \leq \varepsilon e^{-Ck/2}+o_{\varepsilon\rightarrow 0}(\varepsilon)
\]
for $k$ large enough, so that if we set $\epsilon := 1-e^{-C/2}$, we get
\[
\vert A_{k,h}^{q_{h,\delta}}g_h(a)\vert\leq (1-\epsilon)^k.
\]

\end{proof}

\begin{proof}[Proof of Proposition~\ref{lem:grimmettkesten}]
Let $k_0\in \dN$ and $\epsilon >0$ such that (\ref{eqn:fprimeq}) holds with $\delta= \delta_0/2$, and let $\delta_1:=d^{2-{k_0}}\delta_0/2$. Let $C\in (0,(2k_0)^{-1})$. For $k\geq 1$ and $z\in \partial B_{\Td}(\circ, k)$, denote 
$$p_{a,k}:= \dP^{\Td}_a(z\in \Chplus\text{ and }E(z,\delta_1) \leq C k),
$$
which does not depend on the choice of $z$, by cylindrical symmetry of $\Td$. Since $\vert  \partial B_{\Td}(\circ, k)\vert =d(d-1)^{k-1}$, it is enough to prove that 
\begin{equation}\label{eqn:grimmettkestenpzk}
\limsup_{k\rightarrow +\infty} k^{-1}\sup_{a\geq h}\log p_{a,k} < -\log (d-1).
\end{equation}
\noindent
Assume for now that for all $j>i\geq 1$ and $a\geq h$, 
\begin{equation}\label{eqn:Aqhtranslatedintoprobas}
\dP^{\Td}\hspace{-1mm} \left(\hspace{-1mm} z_j\hspace{-1mm} \in \Chplus \text{and }\forall\ell\in [i,j-1],\, z_{\ell}\hspace{-0.5mm} \text{ is not a }\frac{d^{i-j+1}\delta_0}{2}\text{-exit}\bigg| \phid(z_i)=a, z_i\in \Chplus\hspace{-1mm} \right)\hspace{-1mm} \leq \hspace{-1mm} \frac{A_{j-i,h}^{q_{h,\delta_0/2}} g_h(a)}{(d-1)^{j-i}}.
\end{equation}
We will show this technical claim at the end of the proof. 
\\
We proceed to a union bound on the different possibilities for the $\delta_1$-exits on the path $\xi_z$. If $E(z,\delta_1)\leq Ck$, then there exists $u\leq Ck$ and $1\leq i_1<\ldots <i_u\leq k-1$ such that the $\delta_1$-exits on $\xi_z$ are exactly $z_{i_1}, \ldots, z_{i_u}$, where $z_i$ is the vertex on the path $\xi_z$ at height $i$. We first show that 
\begin{equation}\label{eqn:probadeltaexitsarezi}
\sup_{a\geq h}\dP^{\Td}_a(z\in \Ch \text{ and the $\delta_1$-exits of $\xi_z$ are }\{z_{i_1}, \ldots, z_{i_u}\}) \leq \left(\frac{(1-\epsilon)^{k_0}}{(d-1)^{k_0}}\right)^m\leq  \left(\frac{1-\epsilon}{d-1}\right)^{k(1-2Ck_0)}.
\end{equation}
%
%
%
%
Since $u\leq Ck$, we claim that there exist at least $k\frac{1-2Ck_0}{k_0}$ disjoint integer intervals of length $k_0$ in $\{1, \ldots, k\}$ that do not intersect $\{i_1,\ldots, i_u\}$. 
\\
Indeed, write $z_{i_0}:=1$ and $z_{i_{u+1}}:=k$. For every $r\in \{1, \ldots, u\}$,  if $z_{i_{r+1}}> z_{i_r}+k_0$, let $D_r:=\{z_{i_r}+1,\ldots, z_{i_r}+k_0\lfloor (z_{i_{r+1}}-1-z_{i_r})/k_0\rfloor \}$, else let $D_r=\emptyset$. The $D_r$'s are clearly disjoint with $\vert D_r\vert \geq z_{i_{r+1}}-z_{i_r}-k_0$ for all $r$, so that 
\begin{center}
$\vert \cup_{1\leq r\leq u}D_r\vert =\sum_{r=1}^u \vert D_r\vert \geq \sum_{r=1}^u (z_{i_{r+1}}-z_{i_r}-k_0)\geq z_{i_{u+1}}-z_{i_1}-uk_0 \geq k-1-Ckk_0\geq k(1-2Ck_0)$.
\end{center}
Since $\vert D_r\vert $ is a multiple of $k_0$, $D_r$ can be split into $\vert D_r\vert /k_0$ disjoint subsets of $k_0$ consecutive integers. Summing over $r\in \{1, \ldots, u\}$, we get $(\sum_{r=1}^u\vert D_r\vert )/k_0\geq k\frac{1-2Ck_0}{k_0}$ such disjoint integer intervals. This proves our claim.
\\
Denote $I_1, \ldots, I_m$ those intervals for some $m\geq k\frac{1-2Ck_0}{k_0}$, and $\ell_1, \ldots, \ell_m$ their respective smallest element. We have 
\begin{align*}
&\sup_{a\geq h}\dP^{\Td}_a(z\in \Chplus \text{ and the $\delta_1$-exits of $\xi_z$ are }\{z_{i_1}, \ldots, z_{i_u}\})
\\
&\leq \prod_{r=1}^m \sup_{a\geq h}\dP^{\Td}(z_{\ell_r+k_0-1}\hspace{-1mm}\in \Chplus \text{and }\forall l\in\hspace{-1mm} [\ell_r,\ell_r+k_0-2],z_{l}\text{ is not a }\delta_1\text{-exit} \vert z_{\ell_r}\in \Chplus\hspace{-1mm}, \phid(z_{\ell_r})\hspace{-1mm}=\hspace{-1mm}a)
\\
&\leq \left(\frac{\sup_{a\geq h}A_{k_0,h}^{q_{h,\delta_0/2}}g_h(a)}{(d-1)^{k_0}}\right)^m
\end{align*}
by (\ref{eqn:Aqhtranslatedintoprobas}) with $i=\ell_r$ and $j=\ell_r+k_0-1$, and by definition of $\delta_1$. Then, by (\ref{eqn:fprimeq}) and by definition of $\epsilon$, we have
\[
\dP^{\Td}(z\in \Ch \text{ and the $\delta_1$-exits of $\xi_z$ are }\{z_{i_1}, \ldots, z_{i_u}\}) \leq \left(\frac{(1-\epsilon)^{k_0}}{(d-1)^{k_0}}\right)^m\leq  \left(\frac{1-\epsilon}{d-1}\right)^{k(1-2Ck_0)},
\]
and \eqref{eqn:probadeltaexitsarezi} follows.
\\
Second, note that  there are at most $\sum_{u=1}^{Ck}{k\choose u}$ choices for $i_1, \ldots i_u$. Since $C<1/2$, we have
$$\sum_{u=1}^{Ck}{k\choose u}\leq Ck{k\choose Ck}\leq \frac{k^k}{(Ck)^{Ck}(k-Ck)^{(1-C)k}}\leq (C^C(1-C)^{1-C})^{-k}$$ 
for $k$ large enough (depending on $C$), by Stirling's formula. By \eqref{eqn:probadeltaexitsarezi} and a union bound, we obtain
\[
\sup_{a\geq h} p_{a,k}\leq \left(\frac{1-\epsilon}{d-1}\right)^{k(1-2Ck_0)}(C^C(1-C)^{1-C})^{-k}.\leq \left(\frac{1}{d-1}\cdot \frac{1-\epsilon}{C^C(1-C)^{1-C}} \right)^k.
\]
\noindent
Since $\lim_{C\rightarrow 0^+} = \left(\frac{1-\epsilon}{d-1}\right)^{1-2Ck_0} C^{-C}(1-C)^{C-1}=\frac{1-\epsilon}{d-1}<1/(d-1)$, we can choose $C$ small enough such that for all $k$ large enough (depending on $C$), $\sup_{a\geq h}p_{a,k}\leq  \left(\frac{1-\epsilon/2}{d-1}\right)^k$, and (\ref{eqn:grimmettkestenpzk}) follows.
\\

\noindent
\textbf{Proof of \eqref{eqn:Aqhtranslatedintoprobas}:}
We prove this by induction on $j-i$. For the base case $j=i+1$, denote $z'_1, \ldots, z'_{d-2}$ the children of $z_i$ in $\Td$ that are not $z_{i+1}$, and $\cC_{1}, \ldots, \cC_{d-2}$ their respective subtrees in $\Chplus$ (which are possibly empty). Note that
\begin{center}
$\{z_{i+1}\in \Chplus,\, z_i\text{ is not a }\delta_0/2\text{-exit}\}$
\\
$=\{z_i\in \Chplus\}\cap \{ \phid(z_{i+1})\geq h\}\cap (\cap_{\ell=1}^{d-2}\{\cC_{\ell}\text{ is not $\delta_0/2$-transient}\})$. 
\end{center}
The subtrees $\cC_{1}, \ldots, \cC_{d-2}$ are i.i.d.~conditionally on $\phid(z_i)$, so that we obtain as desired:
\begin{align*}
\dP^{\Td}(z_{i+1}\in \Chplus,&\, z_i\text{ is not a }\delta_0/2\text{-exit}\vert\phid(z_i)=a, z_i\in \Chplus )
\\
&=\dP^{\Td}(\phid(z_{i+1})\geq h\vert \phid(z_i)=a)\dP^{\Td}(\cC_{1}\text{ is not $\delta_0/2$-transient}\vert \phid(z_i=a))^{d-2}
\\
&=\dE_Y\left[g_h\left(\frac{a}{d-1}+Y\right)\right]\dE_Y\left[q_{h,\delta_0/2}\left(\frac{a}{d-1}+Y\right)\right]^{d-2}
\\
&=\frac{A_h^{q_{h,\delta_0/2}}g_h(a)}{d-1}.
\end{align*}

\noindent
We proceed to the induction step. If (\ref{eqn:Aqhtranslatedintoprobas}) holds for some value $k\in \dN$ of $j-i$, let $i\in \dN$ and $j=i+k+1$. We have 
\begin{align*}
A_{j-i,h}^{q_{h,\delta_0/2}}g_h(a)&=A_h^{R_h^{k}q_{h,\delta_0/2}}\left(A_{k,h}^{q_{h,\delta_0/2}}g_h\right)(a)
\\
&= (d-1)\dE_Y\left[R_h^{k}q_{h,\delta_0/2}\left(\frac{a}{d-1}+Y\right) \right]^{d-2}\dE_Y\left[A_{k,h}^{q_{h,\delta_0/2}}g_h \left(\frac{a}{d-1}+Y\right)\right].
\end{align*}  
For every $1\leq \ell \leq d-2$, denote $O_{\ell}$ the $k$-offspring of $z'_{\ell}$ in $\Chplus$. For all $z'\in O_{\ell}$,  write $\cC_{z'}$ for the subtree from $z'$. On the one hand, for every $a'\geq h$, and every $\ell\in \{1, \ldots, d-2\}$,
\begin{align*}
R_h^{k}q_{h,\delta_0/2}(a')&=\dE^{\Td}_{a'}\left[\prod_{z'\in \cZ_k^{h,+}}q_{h,\delta_0/2}(z')\right]
\\
&=\dP^{\Td}(\forall  z'\in O_{\ell}, \cC_{z'}\text{ is not $\delta_0/2$-transient}\vert \phid(z'_{\ell})=a').
\end{align*}
Remark that if there exists $z'\in O_{\ell}$ such that $\cC_{z'}$ is $\delta_0/2$-transient, then $z'_{\ell}$ is $d^{-k}\delta_0/2$ transient, since a SRW starting at $z'_{\ell}$ has a probability at least $d^{-k}$ to hit $z'$ before $z'_{\ell}$, and then a probability at least $\delta_0/2$ to stay forever in $\cC_{z'}$ by the Markov property of the SRW. Therefore,
\begin{center} 
$\{\forall  z'\in O_{\ell}, \cC_{z'}\text{ is not $\delta_0/2$-transient}\}\supseteq \{\text{$z'_{\ell}$ is not $d^{-k}\delta_0/2$-transient}\} $. 
\end{center}
Thus, we have $
R_h^{k}q_{h,\delta_0/2}(a')\geq\dP^{\Td}(\text{$z'_{\ell}$ is not $d^{-k}\delta_0/2$-transient }\vert \phid(z'_{\ell})=a'),$
and 
\begin{align*}
\dE_Y\left[R_h^{k}q_{h,\delta_0/2}\left(\frac{a}{d-1}+Y\right) \right]^{d-2}&\geq \dP^{\Td}(\forall 1\leq \ell\leq d-2,\, \text{$z'_{\ell}$ is not $d^{-k}\delta_0/2$-transient } \vert \phid(z_i)=a)
\\
&\geq \dP^{\Td}(z_i\text{ is not a $d^{-k}\delta_0/2$-exit }\vert \phid(z_i)=a).
\end{align*}
On the other hand, by induction hypothesis, for every $a'\geq h$, 
\begin{align*}
\frac{A_{k,h}^{q_{h,\delta_0/2}}g_h(a')}{(d-1)^k}&\geq \dP^{\Td}\left(z_j\in \Ch \text{ and }\forall\ell\in [i+1,j-1],\, z_{\ell}\text{ is not a }d^{1-k}\delta_0/2\text{-exit}\bigg| \phid(z_{i+1})=a'\right)
\\
&\geq \dP^{\Td}\left(z_j\in \Ch \text{ and }\forall\ell\in [i+1,j-1],\, z_{\ell}\text{ is not a }d^{-k}\delta_0/2\text{-exit}\bigg| \phid(z_{i+1})=a'\right).
\end{align*}
Therefore, 
\[
\frac{A_{j-i,h}^{q_{h,\delta_0/2}}g_h(a)}{(d-1)^{k+1}}\geq \dP^{\Td}\left(z_j\in \Ch \text{ and }\forall\ell\in [i,j-1],\, z_{\ell}\text{ is not a }d^{-k}\delta_0/2\text{-exit} \bigg| \phid(z_i)=a\right),
\]
and this concludes the induction.
\end{proof}

\subsection{Looking for an invariant measure for the walk}\label{subsec:invarmeasurenomore}
As mentioned in the introduction, we could not prove the existence of an invariant measure for $\Ch$ rooted at the position $X_k$ of the random walker. In this section, we show more precisely that the method of~\cite{LPPergodic} for Galton-Watson trees does not adapt to the SRW $(X_k)_{k\geq 0}$ on $\Ch$. Let us also mention that the lack of independence in the structure of $\Ch$, and the fact that the distribution of a given subtree depends on the value of the GFF at its root also prevented us to adapt the argument of~\cite{Aidekon} for the biased random walk on Galton-Watson trees (although we do not detail this here).
\\
Let $\cE$ be the set obtained from $\dR^{\Td}$ by identifying every $u\in \dR^{\Td}$ with each $v\in \dR^{\Td}$ that can be obtained from $u$ by swapping two subtrees of $\Td$ whose roots have the same parent. Informally, $\cE$ is the set of real sequences indexed by $\Td$ up to cylindrical symmetry. 
\\
Suppose that there is an invariant measure $\mu_{WALK}$ on $\cE$ that describes the values of the GFF as seen from $X_n$, $n\geq 0$. Precisely, we introduce a random shift operator $\theta$ on $\cE$ similar to \cite{LPPergodic}, that for each fixed real sequence $u=(u_x)_{x\in \Td}$ chooses a uniform neighbour $x$ of $\circ$ such that $u_x\geq h$ and moves the root to $x$, hence $\theta(u)=(u_{\Phi(x)})_{x\in \Td}$ where $\Phi$ is a rooted isomorphism from $\Td$ to itself with $\Phi(\circ)=x$ (if no such neighbour exists, $\theta(u)=u$). An \textbf{isomorphism} between two rooted trees $T$ and $T'$ is a bijection $\Phi:T\rightarrow T'$ preserving the root and the height, and such that for all vertices $x,y\in T$, there is an edge between $x$ and $y$ if and only if there is an edge between $\Phi(x)$ and $\Phi(y)$.

This defines a Markov chain on $\cE$. Then $\mu_{WALK}$ is an invariant measure for this chain. We impose an additional constraint, due to the GFF: 
\begin{equation}\label{eq:muwalkgff}
\begin{split}
&\text{for every $z\in \dR$, on $\cE_z:=\{u\in \cE,\, u_{\circ}=z\}$, $\mu_{WALK}(\cdot \vert u_{\circ}=z)$ coincides with }
\\
&\text{the distribution induced by $(\phid(x))_{x\in \Td}$ conditionally on $\phid(\circ)=z$.}
\end{split}
\end{equation}
Let $\widetilde{\cE}:=\{u\in \cE, \, u_{\circ}\geq h\text{ and }\max_{x:\,\h(x)=1}u_x\geq h\}$ be the subset of $\cE$ where the SRW can make at least one step. Clearly, $\widetilde{\cE}$ is invariant under $\theta$, and we denote $\widetilde{\mu}_{WALK}$ the invariant measure on $\widetilde{\cE}$ induced by $\mu_{WALK}$. 
\\
We show that there is no invariant measure $\widetilde{\mu}_{WALK}$ such that \eqref{eq:muwalkgff} holds and $u_{\circ}$ has a density w.r.t.~$\widetilde{\mu}_{WALK}$. Assume by contradiction that $\widetilde{\mu}_{WALK}$  exists, and denote $\mu_{\circ}$ the marginal distribution of $u_{\circ}$.
\\
Denote $x_1, \ldots, x_{d}$ the neighbours of $\circ$. Under $\mu_{WALK}$, for all $z\geq h$, conditionally on $u_{\circ}=z$, the $u_{x_i}$'s are i.i.d.~with distribution $L_z:=\frac{z}{d-1}+\cN(0,\frac{d}{d-1})$. Denote $L_{z,h}$ the law of a variable $Y\sim L_z$ conditionally on $Y\geq h$.
\\
Going from $\widetilde{\mu}_{WALK}$ to $\mu_{WALK}$ amounts to further conditioning on the fact that at least one of the $u_{x_i}$'s is at least $h$, and a SRW starting at $\circ$ will make its first step to a vertex $x$ such that $u_x\sim L_{z,h}$. By invariance of $\widetilde{\mu}_{WALK}$, this forces $\mu_{\circ}=\int L_{z,h}d\mu_{\circ}(z)$, hence for all $z\geq h$, 
\begin{equation}\label{eq:muwalkfirstconstraint}
\mu_{\circ}(z)=\int_h^{\infty}\mu_{\circ}(t)q_t^{-1} \exp\left(-\frac{d-1}{2d}\left(z-\frac{x}{d-1}\right)^2\right) dt,
\end{equation}
where $q_t:=\sqrt{2\pi d/(d-1)}\dP(Y\geq h)$ for $Y\sim L_t$. If $x_i$ is the vertex where the SRW makes its first step and denoting $x_{i,1}\ldots x_{i,d-1}$ its $d-1$ children, then for all $t\geq h$, conditionally on $u_{x_i}=t$, the $u_{x_{i,j}}$'s are i.i.d.~with distribution $L_{t}$. Again by invariance of $\widetilde{\mu}_{WALK}$, the (unordered) $d$-uplets $(u_{x_1}, \ldots, u_{x_d})$ and $(u_{x_{i,1}}, \ldots, u_{x_{i,d-1}},u_{\circ})$ have the same distribution. Thus by the remarks above \eqref{eq:muwalkfirstconstraint}, conditionally on $u_{x_i}=t$, $u_{\circ}\sim L_{t,h}$. Combining this with \eqref{eq:muwalkfirstconstraint}, we obtain for all $z\geq h$:
\begin{equation}\label{eq:muwalksecondconstraint}
\mu_{\circ}(z)=\int_h^{\infty} \left( \mu_{\circ}(z)q_z^{-1} \exp\left(-\frac{d-1}{2d}\left(t-\frac{z}{d-1}\right)^2\right)  \right) q_{t}^{-1} \exp\left(-\frac{d-1}{2d}\left(z-\frac{t}{d-1}\right)^2\right)dt.
\end{equation}
This simplifies to
\[
q_z=\int_h^{\infty}q_t^{-1} \exp\left(-\frac{d-1}{2d}\left(\left(z-\frac{t}{d-1}\right)^2+\left(t-\frac{z}{d-1}\right)^2\right)\right) dt.
\]
Note that the map $t\mapsto q_t$ from $[h,\infty)$ to $[0,1]$ is non-decreasing, that $q_h>0$ and that $\lim_{t\rightarrow \infty}q_t=\sqrt{2\pi d/(d-1)}$. Hence, for all $z\geq h$, we must have 
\[
1\leq q_z\leq q_h^{-1}\int_h^{\infty}\exp\left(-\frac{d-1}{2d}\left(\left(z-\frac{t}{d-1}\right)^2+\left(t-\frac{z}{d-1}\right)^2\right)\right) dt, 
\]
so that 
\[
I_z:=\int_h^{\infty}\exp\left(-\frac{d-1}{2d}\left(\left(z-\frac{t}{d-1}\right)^2+\left(t-\frac{z}{d-1}\right)^2\right)\right) dt\geq q_h >0.
\]
By expanding the squares in the integral and using that $t^2+z^2\geq 2 \vert tz\vert$ for all $t,z\in \dR$, we get that 
$I_z\leq \int_h^{\infty}\exp\left(-\frac{(d-2)^2}{2d(d-1)}(z^2+t^2)\right)dt$, so that $\lim_{z\rightarrow \infty}I_z=0 $. Therefore, 
 \eqref{eq:muwalksecondconstraint} does not hold. Hence, the desired invariant measure $\widetilde{\mu}_{WALK}$ does not exist.
\\

\end{appendix}

\bibliographystyle{alpha}
\bibliography{gffregulgraph}

\newcommand{\etalchar}[1]{$^{#1}$}
\begin{thebibliography}{DCGRS20}

\bibitem[A{\"i}d14]{Aidekon}
Elie A{\"i}d{\'e}kon.
\newblock Speed of the biased random walk on a galton--watson tree.
\newblock {\em Probability Theory and Related Fields}, 159(3):597--617, Aug
  2014.

\bibitem[AS18]{AS2018}
Angelo Abächerli and Alain-Sol Sznitman.
\newblock Level-set percolation for the gaussian free field on a transient
  tree.
\newblock {\em Ann. Inst. H. Poincaré Probab. Statist.}, 54(1):173--201, 02
  2018.

\bibitem[Av20]{ACregultrees}
Angelo Abächerli and Jiří \v{C}erný.
\newblock {Level-set percolation of the Gaussian free field on regular graphs
  I: regular trees}.
\newblock {\em Electronic Journal of Probability}, 25(none):1 -- 24, 2020.

\bibitem[BACF19]{BenArousCabezasFribergh}
Gérard Ben~Arous, Manuel Cabezas, and Alexander Fribergh.
\newblock Scaling limit for the ant in high-dimensional labyrinths.
\newblock {\em Communications on Pure and Applied Mathematics}, 72(4):669--763,
  2019.

\bibitem[Bar03]{Barlow}
Martin Barlow.
\newblock Random walks on supercritical percolation clusters.
\newblock {\em Annals of Probability}, 32:3024--3084, 2003.

\bibitem[BB07]{BergerBiskup}
Noam Berger and Marek Biskup.
\newblock Quenched invariance principle for simple random walk on percolation
  clusters.
\newblock {\em Probability Theory and Related Fields}, 137(1):83--120, Jan
  2007.

\bibitem[BFGH12]{BenArousFGH}
G{\'e}rard {Ben Arous}, Alexander Fribergh, Nina Gantert, and Alan Hammond.
\newblock {Biased random walks on Galton–Watson trees with leaves}.
\newblock {\em The Annals of Probability}, 40(1):280 -- 338, 2012.

\bibitem[Bow18]{Bowditch}
Adam Bowditch.
\newblock Escape regimes of biased random walks on galton--watson trees.
\newblock {\em Probability Theory and Related Fields}, 170(3):685--768, Apr
  2018.

\bibitem[CFK13]{CroydonFriberghKumagai}
David Croydon, Alexander Fribergh, and Takashi Kumagai.
\newblock Biased random walk on critical galton--watson trees conditioned to
  survive.
\newblock {\em Probability Theory and Related Fields}, 157(1):453--507, Oct
  2013.

\bibitem[CHK18]{CollevechioHolmesKious}
Andrea Collevecchio, Mark Holmes, and Daniel Kious.
\newblock {On the speed of once-reinforced biased random walk on trees}.
\newblock {\em Electronic Journal of Probability}, 23(none):1 -- 32, 2018.

\bibitem[CK08]{CroydonKumagai}
David Croydon and Takashi Kumagai.
\newblock {Random walks on Galton-Watson trees with infinite variance offspring
  distribution conditioned to survive}.
\newblock {\em Electronic Journal of Probability}, 13(none):1419 -- 1441, 2008.

\bibitem[CK23]{GckGffPublished}
Guillaume Conchon-Kerjan.
\newblock {Anatomy of a Gaussian giant: supercritical level-sets of the free
  field on regular graphs}.
\newblock {\em Electronic Journal of Probability}, 28(none):1 -- 60, 2023.

\bibitem[Cro09]{Croydon}
David Croydon.
\newblock {Hausdorff measure of arcs and Brownian motion on Brownian spatial
  trees}.
\newblock {\em The Annals of Probability}, 37(3):946 -- 978, 2009.

\bibitem[DCGRS20]{DuminalGFF2}
Hugo Duminil-Copin, Subhajit Goswami, Pierre-François Rodriguez, and Franco
  Severo.
\newblock Equality of critical parameters for percolation of gaussian free
  field level-sets, 2020.

\bibitem[DGP22]{DrewitzGalloPrevost}
Alexander Drewitz, Gioele Gallo, and Alexis Prévost.
\newblock Generating galton-watson trees using random walks and percolation for
  the gaussian free field, 2022.

\bibitem[DPR18]{DrewitzPrevostRodriguez}
Alexander {Drewitz}, Alexis {Pr{\'e}vost}, and Pierre-Fran{\c{c}}cois
  {Rodriguez}.
\newblock {The Sign Clusters of the Massless Gaussian Free Field Percolate on
  \{Z$^{d}$, d {\ensuremath{\geq}}slant 3\} (and more)}.
\newblock {\em Communications in Mathematical Physics}, page 1398, August 2018.

\bibitem[DPR21]{DPRCritExp}
Alexander Drewitz, Alexis Prévost, and Pierre-François Rodriguez.
\newblock Critical exponents for a percolation model on transient graphs, 2021.

\bibitem[EKM{\etalchar{+}}00]{eisenbaum2000}
Nathalie Eisenbaum, Haya Kaspi, Michael~B. Marcus, Jay Rosen, and Zhan Shi.
\newblock A ray-knight theorem for symmetric markov processes.
\newblock {\em Ann. Probab.}, 28(4):1781--1796, 10 2000.

\bibitem[eL23]{CernyLocher}
Jiří Černý and Ramon Locher.
\newblock Critical and near-critical level-set percolation of the gaussian free
  field on regular trees, 2023.

\bibitem[GK01]{GrimmettKesten}
Geoffrey {Grimmett} and Harry {Kesten}.
\newblock {Random Electrical Networks on Complete Graphs II: Proofs}.
\newblock {\em arXiv Mathematics e-prints}, page math/0107068, July 2001.

\bibitem[HM11]{HairerMattingly}
Martin Hairer and Jonathan~C. Mattingly.
\newblock Yet another look at harris' ergodic theorem for markov chains.
\newblock In Robert Dalang, Marco Dozzi, and Francesco Russo, editors, {\em
  Seminar on Stochastic Analysis, Random Fields and Applications VI}, pages
  109--117, Basel, 2011. Springer Basel.

\bibitem[Kes86]{Kesten}
Harry Kesten.
\newblock Subdiffusive behavior of random walk on a random cluster.
\newblock {\em Annales de l'I.H.P. Probabilit\'es et statistiques},
  22(4):425--487, 1986.

\bibitem[KLNS89]{KahnLinialCovertime}
Jeff Kahn, Nathan Linial, Noam Nisan, and Michael Saks.
\newblock On the cover time of random walks on graphs, 1989.

\bibitem[KN09]{KozmaNachmias}
Gady Kozma and Asaf Nachmias.
\newblock The alexander-orbach conjecture holds in high dimensions.
\newblock {\em Inventiones mathematicae}, 178(3):635--654, Dec 2009.

\bibitem[LPP95]{LPPergodic}
Russell Lyons, Robin Pemantle, and Yuval Peres.
\newblock Ergodic theory on galton—watson trees: speed of random walk and
  dimension of harmonic measure.
\newblock {\em Ergodic Theory and Dynamical Systems}, 15(3):593--619, 1995.

\bibitem[LPP96]{LPP96}
Russell {Lyons}, Robin {Pemantle}, and Yuval {Peres}.
\newblock {Biased random walks on Galton–Watson trees}.
\newblock {\em Probability Theory and Related Fields}, page math/0107068, July
  1996.

\bibitem[Lup16]{Lupu}
Titus Lupu.
\newblock From loop clusters and random interlacements to the free field.
\newblock {\em Ann. Probab.}, 44(3):2117--2146, 05 2016.

\bibitem[Lyo90]{Lyons90}
Russell Lyons.
\newblock Random walks and percolation on trees.
\newblock {\em Ann. Probab.}, 18(3):931--958, 1990.

\bibitem[MT93]{MeynTweedie}
Sean Meyn and Richard Tweedie.
\newblock {\em Markov Chains and Stochastic Stability}.
\newblock Springer Verlag, London, 1993.

\bibitem[Mui22]{Muirhead}
Stephen Muirhead.
\newblock Percolation of strongly correlated gaussian fields ii. sharpness of
  the phase transition, 2022.

\bibitem[Pia98]{Piau}
Didier Piau.
\newblock Théoràme central limite fonctionnel pour une marche au hasard en
  environnement aléatoire.
\newblock {\em Ann. Probab.}, 26(3):1016--1040, 07 1998.

\bibitem[RS13]{RodriguezSznitman}
Pierre-Fran{\c{c}}ois Rodriguez and Alain-Sol Sznitman.
\newblock Phase transition and level-set percolation for the gaussian free
  field.
\newblock {\em Communications in Mathematical Physics}, 320(2):571--601, 2013.

\bibitem[SS04]{SidoraviciusSznitman}
Vladas Sidoravicius and Alain-Sol Sznitman.
\newblock Quenched invariance principles for walks on clusters of percolation
  or among random conductances.
\newblock {\em Probability Theory and Related Fields}, 129(2):219--244, Jun
  2004.

\bibitem[ST16]{SabotTarres}
Christophe Sabot and Pierre Tarr\`es.
\newblock Inverting ray-knight identity.
\newblock {\em Probability Theory and Related Fields}, 165(3):559--580, Aug
  2016.

\bibitem[Szn11]{Sznitman12RIandGFF}
Alain-Sol Sznitman.
\newblock Random interlacements and the gaussian free field.
\newblock {\em The Annals of Probability}, 40, 02 2011.

\bibitem[Szn12]{sznitman2012}
Alain-Sol Sznitman.
\newblock An isomorphism theorem for random interlacements.
\newblock {\em Electron. Commun. Probab.}, 17:9 pp., 2012.

\bibitem[Szn16]{SZ2016}
Alain-Sol Sznitman.
\newblock Coupling and an application to level-set percolation of the gaussian
  free field.
\newblock {\em Electron. J. Probab.}, 21:26 pp., 2016.

\end{thebibliography}

\end{document}